\documentclass[11pt,a4paper]{amsart}

\usepackage{amssymb}
\usepackage{amsmath}
\usepackage{amsfonts}
\usepackage{mathrsfs}
\usepackage{bm}

\usepackage{xfrac}

\usepackage{enumerate}

\usepackage[titletoc]{appendix}

\usepackage{tikz}
\usepackage{tikz-cd}

\usepackage{amsthm}
\theoremstyle{plain} 
\newtheorem{theo}{Theorem}[section]
\newtheorem{lemm}[theo]{Lemma}
\newtheorem{prop}[theo]{Proposition}
\newtheorem{coro}[theo]{Corollary}

\theoremstyle{definition}

\newtheorem{defi}[theo]{Definition} 
\theoremstyle{remark}
\newtheorem{exam}[theo]{Example}
\newtheorem{rema}[theo]{Remark}

\DeclareMathOperator{\biC}{\mathfrak{C}}
\DeclareMathOperator{\IScorr}{\mathfrak{IC}} 
\DeclareMathOperator{\Gr}{\mathfrak{Gr}} 
\DeclareMathOperator{\Corr}{\mathfrak{Corr}} 

\newcommand{\adL}{L}
\newcommand{\cpK}{K}
\newcommand{\U}{\mathcal{U}}
\newcommand{\V}{\mathcal{V}}
\newcommand{\W}{\mathcal{W}}
\newcommand{\X}{\mathcal{X}}
\newcommand{\Y}{\mathcal{Y}}
\newcommand{\E}{\mathcal{E}}
\newcommand{\F}{\mathcal{F}}

\renewcommand{\:}{\colon}
\newcommand{\hi}{\,\mathchar`-\,}
\newcommand{\inv}{^{-1}}
\renewcommand{\d}{^\dag}
\DeclareMathOperator{\id}{id}

\usepackage{xspace}
\newcommand{\Ccorr}{$C^*$-cor\-re\-spon\-dence\xspace}
\newcommand{\Ccorrs}{$C^*$-cor\-re\-spon\-dences\xspace}
\newcommand{\shom}{$*$-ho\-mo\-mor\-phism\xspace}
\newcommand{\shoms}{$*$-ho\-mo\-mor\-phisms\xspace}

\usepackage{braket}
\usepackage{mathtools}
\DeclarePairedDelimiterX{\ip}[2]{\langle}{\rangle}{#1\,\delimsize\vert\,#2}
\newcommand{\rin}[3]{\ip*{#1}{#2}_{#3}}
\newcommand{\lin}[3]{\prescript{}{#3}{\ip*{#1}{#2}}}

\begin{document}

\title[Inverse Sets and Inverse Correspondences]{
	Inverse Sets and Inverse Correspondences Over Inverse Semigroups
}
\author{Tomoki Uchimura}

\maketitle

\begin{abstract}
	In this paper, we introduce notions called inverse set and inverse correspondence over inverse semigroups. 
	These are analogies of Hilbert $C^*$-modules and \Ccorrs in the $C^*$-algebra theory.
	We show that inverse semigroups and inverse correspondences form a bicategory.
	In this bicategory, two inverse semigroups are equivalent if and only if they are Morita equivalent.
\end{abstract}

\setcounter{section}{-1}
\section{Introduction}
\label{Section: Introduction}
The constructions of $C^*$-algebras from inverse semigroups through \'etale groupoids are studied well (see \cite{Pat99}, \cite{Exe08}, or \cite{FKU24} for examples);
\[
\begin{tikzcd}
\text{inverse semigroup} \arrow[r,mapsto] & \text{\'etale groupoid} \arrow[r,mapsto] & \text{$C^*$-algebra}.
\end{tikzcd}
\]
Many researches import notions in the theory of \'etale groupoids and inverse semigroups from the $C^*$-algebra theory. 
Our result is a part of this direction.

The notions of Morita equivalence have been already introduced in the theory of $C^*$-algebras, groupoids and inverse semigroups respectively.
Rieffel introduced and studied the notion of strong Morita equivalence between $C^*$-algebras in \cite{Rie74-1, Rie74-2, Rie76}.
Muhly, Renault, and Williams introduced Morita equivalence for a certain types of groupoids.
They showed that Morita equivalent groupoids produce strong Morita equivalent groupoid $C^*$-algebras in \cite{MRW87, Ren87}.
Steinberg introduced the notion of strong Morita equivalence between inverse semigroups, and showed that strong Morita equivalent inverse semigroups produce Morita equivalent universal groupoids in \cite{Ste11}.

In the $C^*$-algebra theory, there exists the notion called \emph{\Ccorr}.
This is a kind of generalization of both \shoms and Morita equivalences.
Buss, Meyer, and Zhu studied the bicategory $\Corr$ consisting of $C^*$-algebras and non-degenerate \Ccorrs in \cite{BMZ13}.
In the bicategory $\Corr$, two $C^*$-algebras are equivalent if and only if they are Morita equivalent.
Albandik introduced \emph{groupoid correspondences} between \'etale groupoids and the bicategory $\Gr$ consisting of \'etale groupoids and groupoid correspondences in \cite{Alb15}. 
In the bicategory $\Gr$, two groupoids are equivalent if and only if they are Morita equivalent.

However, any notion similar to C*-correspondences or groupoid correspondences in the inverse semigroup theory have not been introduced in our knowledge.
In this paper, we introduce the notion called \emph{inverse correspondence}, which corresponds to $C^*$-correspondence or groupoid correspondence.
This is a kind of generalization of both semigroup homomorphisms and Morita equivalences between inverse semigroups.
More precisely, we define inverse correspondences as follows:
We first introduce \emph{inverse set} $\U$ and \emph{adjointable maps} on $\U$ with the Hilbert $C^*$-module theory in mind.
We show that the set $\adL(\U)$ of all adjointable maps becomes an inverse semigroup (Theorem \ref{Theorem: L(U) is inverse}).
In order to prove this fact, we show that all ``one-rank operators'' $\cpK(\U)$ on $\U$ becomes an inverse subsemigroup of $\adL(\U)$.
For inverse semigroups $S$ and $T$, we define an inverse correspondence as a couple of a right inverse $T$-set $\U$ and a semigroup homomorphism from $S$ to the inverse semigroup $\adL(\U)$.

In the $C^*$-algebra theory, it is well-known that all adjointable operators $\adL(\E)$ on a Hilbert $C^*$-module becomes a $C^*$-algebra, and all compact operators $\cpK(\E)$ becomes a $C^*$-subalgebra of $\adL(\E)$. 
For $C^*$-algebras $A$ and $B$, the $C^*$-correspondence from $A$ to $B$ consists of a right Hilbert $B$-module $\E$ and a $*$-homomorphism from $A$ to the $C^*$-algebra $\adL(\E)$ of adjointable operators on $\E$.
Our result are analogies of these facts.

We show that all inverse semigroups and all non-degenerate inverse correspondences form a bicategory $\IScorr$ in Theorem \ref{Theorem: IScorr}.
Two inverse semigroups are equivalent in our bicategory $\IScorr$ if and only if they are Morita equivalent by Theorem \ref{Theorem: Morita iff equiv}.

In the forthcoming paper \cite{Uch24}, we will construct a bifunctor from $\IScorr$ to $\Gr$.
This bifunctor will generalize the construction from Morita equivalence between inverse semigroups to one between \'etale groupoids in \cite{Ste11}.
We will also construct a bifunctor from $\IScorr$ to $\Corr$, and investigate these bifunctors.

The inverse semigroups $\cpK(\U)$ and $\adL(\U)$ have applications to the inverse semigroup theory in addition to defining the inverse correspondences.
In Section \ref{Section: Multiplier semigroups}, we introduce the \emph{multiplier semigroups} of inverse semigroups, which are analogue to the multiplier algebras of $C^*$-algebras.
We show the existence of multiplier semigroups for all inverse semigroups by using the inverse semigroup $\adL(\U)$.
The inverse semigroup $\cpK(\U)$ can be used to describe the inverse Rees matrix semigroups (Section \ref{Section: Relation to inverse Rees matrix semigroups}).

This paper consists of follows:
In Section \ref{Section: Preliminaries}, we recall the basics of the theory of inverse semigroups.
In Section \ref{Section: Inverse sets and Partial Morita equivalences}, we introduce the notions called inverse sets and partial Morita equivalences. 
These correspond to Hilbert $C^*$-modules and Hilbert bimodules in the $C^*$-algebra theory.
In Section \ref{Section: Adjointable maps between inverse sets}, we introduce adjointable maps on an inverse set and show that all adjointable maps becomes an inverse semigroup.
In Section \ref{Section: Inverse correspondences and their tensor products}, we define tensor product of inverse correspondences. 
In Section \ref{Section: A bicategory IScorr of inverse semigroups}, we introduce the bicategory $\IScorr$ of inverse semigroups and non-degenerate inverse correspondences, and characterize equivalences in $\IScorr$.
In Section \ref{Section: Multiplier semigroups}, we introduce the multiplier semigroups, and show that they exist for all inverse semigroups.
In Section \ref{Section: Relation to inverse Rees matrix semigroups}, we investigate the inverse Rees matrix semigroups in terms of inverse sets.

\section{Preliminaries}
\label{Section: Preliminaries}

We recall definitions and propositions in the inverse semigroup theory.
See \cite{Law98,Pat99}, or \cite{Law23} for more details.

A semigroup $S$ is \emph{regular} if for every $s \in S$ there exists an element $t \in S$ with $sts =s$ and $tst = t$.
Such an element $t$ is called a \emph{generalized inverse} of $s$.
A regular semigroup $S$ is said to be \emph{inverse} if each element has a unique generalized inverse.
For an inverse semigroup $S$, we denote the generalized inverse of $s \in S$ as $s^*$.
It is clear that $s^{**} = s$ for $s\in S$.
We have $(st)^* = t^*s^*$ for $s,t\in S$ by using Proposition \ref{Proposition: e and f commute}.

Let $S$ and $T$ be semigroups.
A map $\theta\: S \rightarrow T$ is a \emph{semigroup homomorphism} if $\theta(ss') = \theta(s)\theta(s')$ for $s,s'\in S$.
If $S$ and $T$ are inverse, $\theta(s^*) = \theta(s)^*$ holds for $s \in S$.

\begin{exam}
	A discrete group is an inverse semigroup which has the unit as a unique idempotent.
\end{exam}

\begin{exam}
	For topological spaces $X$ and $Y$, a \emph{partial homeomorphism} $u$ from $X$ to $Y$ is a homeomorphism from an open subset $D_u$ of $X$ to an open subset $R_u$ of $Y$.
	For a partial homeomorphism $u$ from $X$ to $Y$, we define a partial homeomorphism from $Y$ to $X$ called an inverse of $u$ as the homeomorphism $u\inv$ from $R_u$ to $D_u$.
	We denote the partial homeomorphism by the same symbol $u\inv$.
	For topological spaces $X_1$, $X_2$, $X_3$, and partial homeomorphisms $u_1$ from $X_1$ to $X_2$ and $u_2$ from $X_2$ to $X_3$, we define a composition $u_2u_1$ of $u_1$ and $u_2$ as the partial homeomorphism from $X_1$ to $X_3$ defined by $u_2u_1(x) := u_2(u_1(x))$ for every $x \in D_{u_2u_1} := u_1\inv(D_{u_2})$.
	We denote the set of all partial homeomorphisms from $X$ to $Y$ as $I(X,Y)$ and $I(X,X)$ as $I(X)$.
	The set $I(X)$ becomes an inverse semigroup with respect to the composition of partial homeomorphisms.
\end{exam}

A subset $I$ of a semigroup $S$ is a \emph{left (resp.\ right) ideal} if $st\in I$ (resp.\ $ts \in I$) holds for every $s\in S$ and $t \in I$. 
An \emph{ideal} of $S$ is a left or right ideal of $S$.
A \emph{two-sided ideal} of $S$ is a subset of $S$ which is a left ideal and a right ideal.
An ideal of a semigroup becomes a subsemigroup.
A two-sided ideal of an inverse semigroup becomes an inverse subsemigroup.
We use the following lemma in Lemma \ref{Lemma: left pairing unique}.
\begin{lemm}\label{Lemma: two-sided ideals coincide}
	Let $S$, $T$ be inverse semigroups, and $\theta\: S \rightarrow T$ be a semigroup homomorphism.
	If two-sided ideals $I_1$ and $I_2$ of $S$ satisfy that the restrictions $\theta|_{I_1}$ and $\theta|_{I_2}$ of $\theta$ are injective and $\theta(I_1) = \theta(I_2)$ holds, then we have $I_1 = I_2$.
\end{lemm}
\begin{proof}
	Take $s_1 \in I_1$.
	By assumption, there exists $s_2 \in I_2$ with $\theta(s_1) = \theta(s_2)$.
	Thus we have 
	\[
	\theta(s_1) = \theta(s_1s_1^*s_1) = \theta(s_1s_1^*)\theta(s_1) = \theta(s_1s_1^*)\theta(s_2) = \theta(s_1s_1^*s_2).
	\] 
	Since $I_1$ is a two-sided ideal, we get $s_1s_1^*s_2\in I_1$.
	Thus $s_1 = s_1s_1^*s_2$ because the restriction $\theta|_{I_1}$ of $\theta$ is injective.
	Since $I_2$ is a two-sided ideal, we get $s_1 = s_1s_1^*s_2 \in I_2$.
	Thus $I_1 \subset I_2$.
	We can obtain the reverse inclusion in a similar way.
\end{proof}

An element $s$ of a semigroup $S$ is an \emph{idempotent} if $ss = s$ holds.
The set of all idempotents of $S$ is denoted as $E(S)$.
We can prove the next proposition in a similar way to \cite[Proposition 2.1.1]{Pat99} or \cite[Theorem 3]{Law98}.
\begin{prop}\label{Proposition: e and f commute}
	Let $S$ be a semigroup and $I$ be a two-sided ideal of $S$.
	If $I$ is an inverse subsemigroup of $S$, then for every $e \in E(S)$ and $f \in E(I)$, $ef = fe$ holds.
\end{prop}
\begin{proof}
	Fix $e \in E(S)$ and $f \in E(I)$ arbitrarily.
	Since $I$ is an ideal, $ef$ is an element of $I$.
	This element has a generalized inverse $(ef)^* \in I$ because $I$ is inverse.
	Since $f(ef)^*e \in I$ satisfies
	\begin{align*}
	ef f(ef)^*e ef &= ef(ef)^*ef = ef, \text{ and }\\
	f(ef)^*e  ef f(ef)^*e &= f(ef)^*ef(ef)^*e = f(ef)^*e,
	\end{align*}
	we get $f(ef)^*e = (ef)^*$.
	By simple calculations, we get $f(ef)^*e$ is an idempotent of $I$.
	Thus $(ef)^*$ is an idempotent of $I$.
	Because an idempotent is self-inverse, this implies that $(ef)^* = (ef)^{**} = ef$.
	Thus $ef$ is an idempotent of $I$.
	In a similar way, we get $fe$ is also an idempotent of $I$.
	
	The fact that $ef$ and $fe$ are idempotents of $I$ follows
	\begin{align*}
	(ef)(fe)(ef) &= efef = ef \text{ and }\\
	(fe)(ef)(fe) &= fefe = fe.
	\end{align*} 
	This implies that $(ef)^* = fe$.
	Thus we get $ef = (ef)^* = fe$.
\end{proof}

\begin{theo}\label{Theorem: regular and inverse}
	A regular semigroup $S$ is inverse if and only if all idempotents of $S$ commute.
\end{theo}
\begin{proof}
	The only if part follows from Proposition \ref{Proposition: e and f commute}.
	See \cite[Theorem 3]{Law98} for a proof of the if part.
\end{proof}

Let $S$ be an inverse semigroup.
The following lemma is well-known:
\begin{lemm}\label{Lemma: order}
	Let $s,t$ be elements of $S$. 
	The following are equivalent;
	\begin{enumerate}[(i)]
		\item $s = ts^*s$,
		\item there exists $e \in E(S)$ with $s = te$,
		\item $s = ss^*t$,
		\item there exists $f \in E(S)$ with $s = ft$.
	\end{enumerate}
\end{lemm}
\begin{proof}
	We see that (i) implies (ii) and that (iii) implies (iv) trivially.
	Let us show that (ii) implies (iii). 
	For $s,t\in S$ and $e \in E(S)$ with $s = te$, we get 
	\begin{align*}
	ss^*t = (te)(te)^*t = tet^*t = tt^*te = te = s.
	\end{align*}
	The third equal follows from Theorem \ref{Theorem: regular and inverse}.
	We can see that (iv) implies (i) similarly.
\end{proof}

We define a binary relation on $S$ by declaring that $s \leq t$ if and only if $s$ and $t$ satisfies one (and hence all) of the conditions in Lemma \ref{Lemma: order}.
We can check easily that this relation $\leq$ becomes a partial order on $S$.
We can show that $s \leq t$ implies $s's\leq s't$ and $ss' \leq ts'$ for every $s' \in S$, and that $s \leq t$ implies $s^* \leq t^*$.
We prove an analogy of Lemma \ref{Lemma: order} in Proposition \ref{Proposition: order}.

\begin{lemm}\label{Lemma: right cancel}
	Let  $s_1,s_2$ be elements of $S$.
	\begin{enumerate}[(i)]
		\item If $s_1s_1^* = s_1s_2^* =s_2s_2^*$, then $s_1 = s_2$.
		\item  If $s_1s = s_2s$ for all $s \in S$, then $s_1 = s_2$.
	\end{enumerate}
\end{lemm}
\begin{proof}
	\begin{enumerate}[(i)]
		\item Take $s_1,s_2 \in S$ with $s_1s_1^* = s_1s_2^* = s_2s_2^*$.
		We have
		\begin{align*}
		s_2^*s_1s_2^* &= s_2^*s_2s_2^* = s_2^*,\\
		s_1s_2^*s_1 &= s_1s_1^*s_1 = s_1.
		\end{align*}
		Since both $s_1$ and $s_2$ are the generalized inverse of $s_2^*$, we obtain $s_1 = s_2$.
		
		\item Let $s_1,s_2$ be elements of $S$ such that $s_1s = s_2s$ for all $s\in S$.
		Taking $s_1^*$ and $s_2^*$ as $s$, we get $s_1s_1^* = s_2s_1^*$ and $s_1s_2^* = s_2s_2^*$.
		The first equation implies $s_1s_1^* = s_1s_2^*$ by taking generalized inverses.
		Thus $s_1 = s_2$ holds by (i). \qedhere
	\end{enumerate}
\end{proof}

\section{Inverse sets and Partial Morita equivalences}
\label{Section: Inverse sets and Partial Morita equivalences}

In this section, we introduce notions of \emph{inverse set} and \emph{partial Morita equivalence}. 
These correspond to Hilbert $C^*$-modules and Hilbert bimodules in the $C^*$-algebra theory.
Steinberg introduced the notion of Morita contexts in \cite{Ste11}.
This corresponds to imprimitivity bimodule in the theory of $C^*$-algebras.
Partial Morita equivalences are generalization of Morita contexts.

\subsection{Inverse sets}

Let $S$ be an inverse semigroup.
\begin{defi}
	A \emph{left $S$-set} $\U$ is a set $\U$ equipped with a left $S$-action, that is, a map $S\times \U \rightarrow \U; (s,u)\mapsto su$ such that $s'(su) = (s's)u$ for every $s,s'\in S$ and $u \in \U$.
	A \emph{right $S$-set} $\U$ is defined in a similar way.
\end{defi}

\begin{defi}\label{Definition: inverse set}
	A \emph{left regular $S$-set} is a left $S$-set equipped with a map $\lin{\cdot}{\cdot}{\U} \: \U \times \U \rightarrow S$ called a \emph{left regular pairing} on $\U$ which satisfies that 
	\begin{enumerate}[(L-i)]
		\item $\lin{ su }{ u' }{\U} = s\lin{ u }{ u' }{\U}$,
		\item $\lin{ u }{ u' }{\U}^* = \lin{ u' }{ u }{\U}$,
		\item $\lin{ u }{ u }{\U} u = u$,
	\end{enumerate}
	for every $u,u' \in \U$ and $s \in S$. 
	A \emph{left inverse $S$-set} is a left regular $S$-set whose left regular pairing satisfies that
	\begin{enumerate}
		\item[(L-iv)] $\lin{ u }{ u' }{\U} u = u$ and $\lin{ u' }{ u }{\U} u' = u'$ imply $u = u'$
	\end{enumerate}
	for every $u,u'\in\U$.
	We call a left regular pairing with (L-iv) as a \emph{left inverse pairing}, or just a left pairing.
	
	A \emph{right regular $S$-set} $\U$ is a right $S$-set with a map $\rin{\cdot}{\cdot}{\U} \: \U \times \U \rightarrow S$ called a \emph{right regular pairing} on $\U$ which satisfies that 
	\begin{enumerate}[(R-i)]
		\item $\rin{ u }{ u's }{\U} = \rin{ u }{ u' }{\U} s$,
		\item $\rin{ u }{ u' }{\U}^* = \rin{ u' }{ u }{\U}$,
		\item $u \rin{ u }{ u }{\U} = u$,
	\end{enumerate}
	for every $u, u' \in \U$ and $s \in S$. 
	A \emph{right inverse $S$-set} $\U$ is a right regular $S$-set whose right regular pairing satisfies that
	\begin{enumerate}
		\item[(R-iv)] $u \rin{ u' }{ u }{\U} = u$ and $u' \rin{ u }{ u' }{\U} = u'$ imply $u = u'$
	\end{enumerate}
	for every $u,u'\in\U$.
	We call a right regular pairing with (R-iv) as a \emph{right inverse pairing}, or just a right pairing.
\end{defi}

As a first example, we regard an inverse semigroup $S$ as a left (and right) inverse $S$-set.
\begin{exam}\label{Example: inverse semigroup}
	We define a left action of $S$ on $S$ as the multiplication from the left side.
	We set a map $\lin{\cdot}{\cdot}{S}\: S \times S \rightarrow S$ as $\lin{s'}{s}{S} := s's^*$ for every $s,s'\in S$.
	It is clear that this map satisfies (L-i) and (L-ii).
	The map $\lin{\cdot}{\cdot}{S}$ satisfies (L-iii) by the definition of the generalized inverse and satisfies (L-iv) since $S$ is inverse.
	Thus $S$ is a left inverse $S$-set with respect to the above structures.
	We can regard $S$ as also a right inverse $S$-set as follows:
	We set a right action of $S$ on $S$ as the multiplication from the right side and define a map $\rin{\cdot}{\cdot}{S}\: S \times S \rightarrow S$ by $\rin{s}{s'}{S} := s^*s'$ for every $s,s'\in S$. 
\end{exam}

\begin{rema}
	Steinberg imposed (L-iii) and (R-iii) when defining Morita contexts in \cite{Ste11}.
	With Example \ref{Example: inverse semigroup} in mind, we receive (L-iii) and (R-iii) as kinds of regularity.
	There we impose (L-iv) and (R-iv) as corresponding to the uniqueness of the generalized inverse.
	These conditions imply many important properties as seen in Lemma \ref{Lemma: non-degenerate} and Proposition \ref{Proposition: R-iv}.
\end{rema}

Let $S$ be an inverse semigroup.

\begin{lemm}\label{Lemma: left pairing basics}
	For a left regular $S$-set $\U$,
	\begin{enumerate}[(i)]
		\item $\lin{ u }{ u }{\U} \in E(S)$
		\item $\lin{ u }{ su' }{\U} = \lin{ u }{ u' }{\U} s^*$
	\end{enumerate}
	hold for every $u,u'\in \U$, $s \in S$.
\end{lemm}
\begin{proof}
	For every $u \in \U$,
	\[
	\lin{u}{u}{\U} \lin{u}{u}{\U} = \lin{ \lin{u}{u}{\U} u }{ u }{\U} = \lin{u}{u}{\U}
	\]
	holds by (L-i) and (L-iii). 
	For $u,u'\in\U$, 
	\[
	\lin{ u }{ su' }{\U} = \lin{ su' }{ u }{\U}^* = \big(s\lin{ u' }{ u }{\U}\big)^* = \lin{ u' }{ u }{\U}^*s^* = \lin{ u }{ u' }{\U}s^*
	\]
	holds by (L-ii) and (L-i).
\end{proof}

\begin{lemm}\label{Lemma: non-degenerate}
	Let $\U$ be a left inverse $S$-set.
	For $u,u'\in\U$, the following hold:
	\begin{enumerate}[(i)]
		\item If $\lin{ u }{ u }{\U} = \lin{ u }{ u' }{\U} = \lin{ u' }{ u' }{\U}$, then $u = u'$.
		\item If $\lin{ u }{ u'' }{\U} = \lin{ u' }{ u'' }{\U}$ holds for all $u'' \in \U$, then $u = u'$.
	\end{enumerate}
\end{lemm}
\begin{proof}
	Take $u,u' \in \U$ with $\lin{ u }{ u }{\U} = \lin{ u' }{ u' }{\U} = \lin{ u }{ u' }{\U}$.
	By taking the generalized inverse, we obtain $\lin{ u }{ u }{\U} = \lin{ u' }{ u }{\U}$.
	We see that
	\begin{align*}
	u \lin{ u' }{ u }{\U} = u \lin{ u }{ u }{\U} = u \text{ and } u' \lin{ u }{ u' }{\U} = u' \lin{ u' }{ u' }{\U} = u'.
	\end{align*}
	These imply $u = u'$ by (L-iv) in Definition \ref{Definition: inverse set}.
	
	Assume $u,u' \in \U$ satisfy $\lin{ u }{ u'' }{\U} = \lin{ u' }{ u'' }{\U}$ for all $u'' \in \U$.
	Taking $u$ and $u'$ as $u''$, we get $\lin{u}{u}{\U} = \lin{u'}{u}{\U}$ and $\lin{u}{u'}{\U} = \lin{u'}{u'}{\U}$ respectively.
	As wee see in the above argument, these imply $u = u'$.
\end{proof}

We will show that for a left regular $S$-set $\U$, (i) in the lemma above implies that $\U$ is inverse in Proposition \ref{Proposition: R-iv}.

\begin{defi}
	For every regular $S$-set $\U$, we set the subset $\lin{\U}{\U}{\U}$ of $S$ as $\{ \lin{u}{u'}{\U} \mid u,u'\in\U \}$.
\end{defi}

For a regular $S$-set $\U$, the subset $\lin{\U}{\U}{\U}$ is a two-sided ideal (and especially an inverse subsemigroup) of $S$ by (L-i) in Definition \ref{Definition: inverse set} and Lemma \ref{Lemma: left pairing basics} (ii).

\begin{defi}
	For a left regular $S$-set $\U$, the left pairing on $\U$ is said to be \emph{left full} if $\lin{\U}{\U}{\U} = S$.
	In this case, we also say that $\U$ is left full.
	We define a term \emph{right full} for a right regular $S$-set in a similar way.
\end{defi}

Let $\U$ and $\V$ be left regular $S$-sets.
\begin{defi}
	Let $\sigma$ be a map from $\U$ to $\V$.
	\begin{enumerate}[(i)]
		\item A map $\sigma$ is a \emph{left $S$-map} if $\sigma(su) = s\sigma(u)$ holds for every $u \in \U$ and $s \in S$.
		\item A map $\sigma$ is \emph{left pairing preserving} if $\lin{ \sigma(u) }{ \sigma(u') }{\V} = \lin{ u }{ u' }{\U}$ holds for every $u,u'\in \U$.
	\end{enumerate}
	For a map between right regular $S$-sets, we similarly define that it is a \emph{right $S$-map} and is \emph{right pairing preserving}.
\end{defi}

\begin{defi}
	A left pairing preserving left $S$-map $\sigma \: \U \rightarrow \V$ is an \emph{isomorphism} if there exists a left pairing preserving left $S$-map $\tau \: \V \rightarrow \U$ such that $\tau\circ\sigma$ is the identity map $1_\U$ for $\U$ and $\sigma\circ\tau$ is the identity map $1_\V$ for $\V$.
	Two left regular $S$-sets are \emph{isomorphic} if there exists an isomorphism between them.
	For right regular $S$-sets $\U$, $\V$ and a map $\sigma\: \U \rightarrow \V$, we similarly define that $\sigma$ is an isomorphism and that $\U$ and $\V$ are isomorphic.
\end{defi}

\begin{lemm}\label{Lemma: isom btwn regular set}
	Let $\sigma\: \U \rightarrow \V$ be a left pairing preserving left $S$-map.
	A map $\tau\: \V \rightarrow \U$ which satisfies $\sigma\circ\tau = 1_\V$ and $\tau\circ\sigma = 1_\U$ becomes a left pairing left $S$-map.
\end{lemm}
\begin{proof}
	For every $s \in S$ and $v,v' \in \V$, we have
	\begin{align*}
	s\tau(v) &= 
	\tau(\sigma(s\tau(v))) = 
	\tau(s\sigma(\tau(v))) = 
	\tau(sv),\\
	\lin{ \tau(v) } { \tau(v') }{\U} &= 
	\lin{ \sigma(\tau(v)) }{ \sigma(\tau(v')) }{\V} = 
	\lin{ v }{ v' }{\V}.
	\end{align*}
	Thus $\tau$ is a left pairing preserving left $S$-map.
\end{proof}

\begin{lemm}\label{Lemma: pairing preserving S-map}
	If $\V$ is inverse, then a left pairing preserving map $\sigma\: \U \rightarrow \V$ is a left $S$-map.
\end{lemm}
\begin{proof}
	For $u \in \U$ and $s \in S$,
	\begin{align*}
	\lin{ \sigma(su) }{ \sigma(su) }{\V} &= \lin{ su }{ su }{\U},\\
	\lin{ s\sigma(u) }{ \sigma(su) }{\V} &= s\lin{ \sigma(u) }{ \sigma(su) }{\V} = s \lin{ u }{ su }{\U} = \lin{ su }{ su }{\U}, \text{ and } \\
	\lin{ s\sigma(u) }{ s\sigma(u) }{\V} &= s\lin{ \sigma(u) }{ \sigma(u) }{\V} s^* = s \lin{ u }{ u }{\U} s^* = \lin{ su }{ su }{\U}
	\end{align*}
	hold. 
	This implies $\sigma(su) = s\sigma(u)$ by Lemma \ref{Lemma: non-degenerate} (i).
\end{proof}

\begin{lemm}\label{Lemma: pairing preserving injective}
	If $\U$ is inverse, then a left pairing preserving map $\sigma\: \U \rightarrow \V$ is injective.
\end{lemm}
\begin{proof}
	Let $\sigma\: \U \rightarrow \V$ be a left pairing preserving map.
	For every $u,u' \in \U$ with $\sigma(u) = \sigma(u')$, 
	$\lin{ u }{ u }{\U} = \lin{ \sigma(u) }{ \sigma(u) }{\V}$, $\lin{ u }{ u' }{\U} = \lin{ \sigma(u) }{ \sigma(u') }{\V}$, and $\lin{ u' }{ u' }{\U} = \lin{ \sigma(u') }{ \sigma(u') }{\V}$ are the same element of $S$.
	This implies $u = u'$ by Lemma \ref{Lemma: non-degenerate} (i).
\end{proof}

\begin{coro}\label{Corollary: isom}
	Let $\U$ and $\V$ be left inverse $S$-set and $\sigma\: \U \rightarrow \V$ be a map.
	The following are equivalent:
	\begin{enumerate}[(i)]
		\item $\sigma$ is an isomorphism.
		\item $\sigma$ is a left pairing preserving left $S$-map, and there exists a map $\tau\:\V\rightarrow\U$ which satisfies $\sigma\circ\tau = 1_\V$ and $\tau\circ\sigma = 1_\U$.
		\item $\sigma$ is left pairing preserving and surjective.
	\end{enumerate}
\end{coro}
\begin{proof}
	It is clear that (i) $\Rightarrow$ (ii) $\Rightarrow$ (iii).
	By Lemma \ref{Lemma: isom btwn regular set}, we get (ii) implies (i).
	It follows from Lemma \ref{Lemma: pairing preserving S-map} and \ref{Lemma: pairing preserving injective} that (iii) implies (ii).
\end{proof}

We can similarly show right versions of Lemma \ref{Lemma: isom btwn regular set}, \ref{Lemma: pairing preserving S-map}, \ref{Lemma: pairing preserving injective} and Corollary \ref{Corollary: isom}.

\subsection{Partial Morita equivalences}
Let $S$ and $T$ be inverse semigroups.

\begin{defi}
	A \emph{$S \hi T$ biset} $\U$ is a set $\U$ equipped with a left action of $S$ and a right action of $T$ which satisfy $s(ut) = (su)t$ for $s \in S$, $t \in T$ and $u \in \U$.
\end{defi}

\begin{defi}\label{Definition: partial Morita context}
	A \emph{partial Morita equivalence} from $S$ to $T$ is a $S \hi T$ biset $\U$ equipped with a left regular pairing $\lin{\cdot}{\cdot}{\U}\: \U \times \U \rightarrow S$ and a right regular pairing $\rin{\cdot}{\cdot}{\U}\: \U \times \U \rightarrow T$ which satisfy $\lin{ u }{ u'}{\U} u'' = u \rin{ u' }{ u'' }{\U}$ for all $u,u',u'' \in \U$.
\end{defi}

\begin{lemm}\label{Lemma: partial Morita context is a left and right inverse set}
	A partial Morita equivalence is a left inverse $S$-set and a right inverse $T$-set.
\end{lemm}
\begin{proof}
	We show that $\lin{\cdot}{\cdot}{\U}$ satisfies the condition (L-iv).
	We can check that $\rin{\cdot}{\cdot}{\U}$ satisfies the condition (R-iv) in a similar way.
	Fix elements $u,u' \in \U$ with $\lin{ u }{ u' }{\U} u = u$ and $\lin{ u' }{ u }{\U} u' = u'$.
	Since
	\[
	\lin{ u }{ u' }{\U} \lin{ u }{ u' }{\U} = \lin{ \lin{ u }{ u' }{\U} u }{ u' }{\U} u = \lin{ u }{ u' }{\U}
	\]
	holds, we get $\lin{ u }{ u' }{\U} \in E(S)$.
	Hence $\lin{ u }{ u' }{\U} = \lin{ u' }{ u }{\U}$ holds.
	This implies that 
	\begin{align*}
	u 
	&= \lin{ u }{ u' }{\U} u 
	= \lin{ u' }{ u }{\U} u
	= u' \rin{ u }{ u }{\U},{ and }\\
	u' 
	&= \lin{ u' }{ u }{\U} u'
	= \lin{ u }{ u' }{\U} u'
	= u \rin{ u' }{ u' }{\U}
	\end{align*}
	hold.
	We set $e := \rin{ u }{ u }{\U}$ and $e':= \rin{ u' }{ u' }{\U}$.
	By Lemma \ref{Lemma: left pairing basics}, we get $e,e'\in E(T)$.
	We remark that $u' = u'\rin{ u' }{ u' }{\U} = u'e'$ by (R-iii).
	Thus we get 
	\[
	u 
	= u' e
	= u' e' e
	= u' e e'
	= u e'
	= u'. \qedhere
	\]
\end{proof}

\begin{defi}\label{Definition: Morita equiv}
	A \emph{Morita equivalence} from $S$ to $T$ is a partial Morita equivalence $\U$ from $S$ to $T$ which is full as both a left regular $S$-set and a right regular $T$-set.
	Two inverse semigroup $S$ and $T$ are \emph{Morita equivalent} if there exists a Morita equivalence from $S$ to $T$.
\end{defi}

\begin{rema}
	In \cite[Definition 2.1]{Ste11}, Steinberg called a Morita equivalence as an \emph{equivalence bimodule}, and a tuple of two inverse semigroups and a Morita equivalence between them as a \emph{Morita context}.
	Two Morita equivalent inverse semigroups in our term are said to be \emph{strongly Morita equivalent} in \cite{Ste11,FLS11}.
	According to \cite{FLS11}, strong Morita equivalence is equivalent to the three notions; topos Morita equivalence, semigroup Morita equivalence, and enlargement Morita equivalence between inverse semigroups. 
\end{rema}

\begin{rema}
	Steinberg showed that strong Morita equivalence is an equivalence relation.
	In \cite[Proposition 2.5]{Ste11}, the transitivity is proved by introducing the tensor products of equivalence bimodules.
	We will introduce the notion of tensor product for inverse sets in Section \ref{Section: Inverse correspondences and their tensor products}.
	This is a generalization of the one introduced by Steinberg.
	We check that being Morita equivalent is reflexive, transitive, and symmetry in Example \ref{Example: enlargement}, \ref{Example: tensor Morita equiv}, and p.27 respectively.
\end{rema}

For a partial Morita equivalence $\U$ from $S$ to $T$, we get a Morita equivalence $\U$ from the subsemigroup $\lin{\U}{\U}{\U}$ of $S$ to the subsemigroup $\rin{\U}{\U}{\U}$ of $T$.
Our term ``partial Morita equivalence'' is derived from this fact.

\begin{exam}\label{Example: enlargement}
	Let $S$ be an inverse semigroup and $T$ be an inverse subsemigroup of $S$.
	If $T$ satisfies $TST = T$, then $TS$ is a left full partial Morita equivalence from $T$ to $S$ with respect to the left $T$-action defined as the multiplication from the left side, the left pairing $\lin{ u_1 }{ u_2 }{TS} := u_1u_2^* \in T$, the right $S$-action defined as the multiplication from the right side, and the right pairing $\rin{ u_1 }{ u_2 }{TS} := u_1^*u_2 \in S$ for $u_1,u_2\in TS$.
	Moreover, if the subsemigroup $T$ satisfies $STS = S$, then $TS$ is a Morita equivalence from $T$ to $S$.
	Especially, an inverse semigroup $S$ can be regarded as a Morita equivalence from $S$ to $S$.
\end{exam}

\begin{rema}
	For an inverse subsemigroup $T$ of $S$ with $TST=T$ and $STS=S$, Lawson called $S$ an \emph{enlargement} of $T$ in \cite{Law96}.
	Steinberg gave an enlargement as a first example of a Morita context in \cite[Proposition 2.2]{Ste11}.
\end{rema}

\subsection{Examples of inverse sets and partial Morita equivalences}

\begin{exam}\label{Example: group}
	Let $G$ be a group.
	We can check that the empty set is an inverse $G$-set trivially.
	The left inverse $G$-set $G$ in Example \ref{Example: inverse semigroup} is another example.
	We show that every left inverse $G$-set $\U$ is either the empty set or isomorphic to the left inverse $G$-set $G$.
	
	Let $\U$ be a non-empty left inverse $G$-set $\U$.
	Fix $u_0 \in \U$ arbitrarily.
	We define a map $\sigma \: G \rightarrow \U\: g \mapsto gu_0$.
	Let us show that this map is surjective.
	For every $u \in \U$, we set $g_u:= \lin{ u }{ u_0 }{\U}$.
	The elements $\lin{ g_uu_0 }{ g_uu_0 }{\U}$, $\lin{ g_uu_0 }{ u }{\U}$, and $\lin{ u }{ u }{\U}$ are idempotents of $G$.
	Hence these are the identity of $G$.
	By Lemma \ref{Lemma: non-degenerate} (i), we get $g_uu_0 = u$.
	The map $\sigma$ is a left pairing preserving because we can see
	\[
	\lin{ g'u_0 }{ gu_0 }{\U} = g'\lin{ u_0 }{ u_0 }{\U}g\inv = g' g\inv = \lin{g'}{g}{G}
	\]
	for every $g,g'\in G$.
	By Lemma \ref{Corollary: isom}, $\sigma$ is an isomorphism between left inverse $G$-sets $G$ and $\U$.
\end{exam}

\begin{exam}\label{Example: idempotent}
	Let $E$ be an inverse semigroup that consists of only idempotents.
	We analyze left inverse $E$-sets.
	We first show that the following pair produces a left inverse $E$-set:
	A pair $(\{ \U_e \}_{e\in E}, \{\sigma_{e,f}\}_{e,f\in E})$ consists of 
	\begin{itemize}
		\item a family of sets $\{\U_e\}_{e\in E}$, and 
		\item a family of maps $\{\sigma_{e,f}\: \U_f \rightarrow \U_e \mid e,f\in E \text{ with } e \leq f \}$.
	\end{itemize}
	We assume that this pair satisfies the following conditions:
	\begin{enumerate}[(i)]
		\item $\sigma_{e,e}$ is the identity map on $\U_e$ for every $e\in E$.
		\item $\sigma_{e_1,e_2}\sigma_{e_2,e_3} = \sigma_{e_1,e_3}$ for every $e_1,e_2,e_3\in E$ with $e_1\leq e_2$ and $e_2\leq e_3$.
		\item For every $e_1,e_2\in E$, $u_1 \in \U_{e_1}$, and $u_2 \in \U_{e_2}$, there exists the largest element $e \in E$ which satisfies $e \leq e_1, e_2$ and $\sigma_{e,e_1}(u_1) = \sigma_{e,e_2}(u_2)$, where ``largest'' means that every $e' \in E$ with $e' \leq e_1,e_2$ and $\sigma_{e',e_1}(u_1) = \sigma_{e',e_2}(u_2)$ satisfies $e' \leq e$. 
	\end{enumerate}
	For such a pair $(\{ \U_e \}_{e\in E}, \{\sigma_{e,f}\}_{e,f\in E})$, the set $\U := \bigsqcup_{e \in E} \U_e$ becomes a left inverse $E$-set with respect to
	\begin{itemize}
		\item a left action of $E$ on $\U$ defined as $fu_1 := \sigma_{fe_1,e_1}(u_1)$, and
		\item a left pairing on $\U$ defined as $\lin{ u_1 }{ u_2 }{\U} := e$ appeared in (iii)
	\end{itemize}
	for $e_1,e_2,f\in E$, $u_1\in \U_{e_1}$, and $u_2 \in \U_{e_2}$.
	
	Let us show that $\U$ is a left inverse $E$-set.
	The map $E \times \U \rightarrow \U; (f,u) \mapsto fu$ is a left action by the condition (ii).
	Since the condition (iii) is symmetry for $u_1$ and $u_2$, the map $\lin{\cdot}{\cdot}{\U}\: \U\times\U\rightarrow E$ satisfies (L-ii).
	The condition (L-iii) follows from the condition (i).
	We check the condition (L-i).
	Take $e_1,e_2,f \in E$, $u_1 \in \U_{e_1}$ and $u_2 \in \U_{e_2}$.
	Put $e := \lin{fu_1}{u_2}{\U}$ and $e':=\lin{u_1}{u_2}{\U}$.
	These satisfy that $e \leq fe_1,e_2$, $\sigma_{e,fe_1}(fu_1) = \sigma_{e,e_2}(u_2)$, $e' \leq e_1,e_2$, and $\sigma_{e',e_1}(u_1) = \sigma_{e',e_2}(u_2)$.	
	We get $fe' \leq fe_1$, $fe' \leq e' \leq e_2$, and $fe'fu_1 = fe'u_1 = fe'u_2$.
	By the maximality of $e$, we have $fe' \leq e$. 
	We get $e \leq e_2$, $e\leq fe_1 \leq e_1$, and $eu_1 = efu_1 = eu_2$.
	The last equation follows from $e\leq fe_1 \leq f$.
	By the maximality of $e'$, we have $e \leq e'$.
	This shows $e = fe \leq fe'$.
	Thus we conclude $e = fe'$. 
	This means that $\lin{fu_1}{u_2}{\U} = f\lin{u_1}{u_2}{\U}$.
	We check the condition (L-iv).
	Take $e_1,e_2\in E$, $u_1 \in \U_{e_1}$, and $u_2 \in \U_{e_2}$ with $\lin{u_1}{u_2}{\U}u_1 = u_1$ and $\lin{u_2}{u_1}{\U}u_2 = u_2$.
	Put $e := \lin{u_1}{u_2}{\U}$. 
	The assumptions mean $eu_1 = u_1$ and $eu_2 = u_2$.
	By the definition of $e$, we get $eu_1 = eu_2$.
	Thus $u_1 = u_2$ holds.
	We have shown that $\U$ is a left inverse $E$-set.
	
	Conversely, we can show that every inverse $E$-set $\U$ is constructed from such a pair $(\{ \U_e \}_{e\in E}, \{\sigma_{e,f}\}_{e,f\in E})$.
	Let $\U$ be a left inverse $E$-set.
	For every $e \in E$, we define
	\begin{itemize}
		\item $\U_e := \{ u\in\U \mid \lin{u}{u}{\U} = e \}$, and 
		\item $\sigma_{e,f}\: \U_f \rightarrow \U_e; u \mapsto eu$ for every $e,f \in E$ with $e \leq f$.
	\end{itemize}
	We can easily verify that these satisfy the conditions (i)-(iii), and this pair produces the given inverse $E$-set $\U$.
\end{exam}

\begin{exam}\label{Example: inverse subset}
	Let $S$ be an inverse semigroup and $\U$ be a left inverse $S$-set.
	For an inverse subsemigroup $T$ of $S$ with $TST = T$, we set $T\U:=\{ tu \in \U \mid t\in T, u \in \U \}$.
	The set $T\U$ becomes a left inverse $T$-set.
	If $\U$ is left full, then so is $T\U$.
\end{exam}

\begin{exam}\label{Example: direct sum}
	Let $S$ be an inverse semigroup with $0$, where $0 \in S$ is the element such that $0s = s0 = 0$ for every $s \in S$.
	For a left inverse $S$-set $\U$ and every $u_1,u_2,u_3\in \U$, we get 
	\[
	\lin{ 0u_1 }{ u_3 }{\U} = 0 \lin{ u_1 }{ u_3 }{\U} = 0 = 0 \lin{ u_2 }{ u_3 }{\U} = \lin{ 0u_2 }{ u_3 }{\U}.
	\]
	Thus we get $0u_1 = 0u_2$ by Lemma \ref{Lemma: non-degenerate} (ii).
	This implies that $0u$ is the same element for all $u\in \U$.
	We denote the element $0u$ as $0_{\U}$.
	
	For two non-empty left inverse $S$-sets $\U$ and $\U'$, we define a set $\U \oplus \U'$ as the disjoint union of $\U$ and $\U'$ identified $0_{\U}$ and $0_{\U'}$.
	We can define a left action $S$ on $\U \oplus \U'$ by using left actions on $\U$ and $\U'$.
	A left pairing on $\U \oplus \U'$ is defined as 
	\[
	\lin{u_1}{u_2}{\U \oplus \U'} :=
	\begin{cases}
	\lin{ u_1 }{u_2}{\U}& u_1,u_2\in \U,\\
	\lin{ u_1 }{u_2}{\U'}& u_1,u_2\in \U',\\
	0 & otherwise.\\
	\end{cases}
	\]
	The set $\U \oplus \U'$ becomes a left inverse $S$-set with respect to the above structures.
\end{exam}

\begin{exam}\label{Example: partial homeo}
	We define a left action of $I(Y)$ on $I(X,Y)$ by the composition from the left side and a left pairing on $I(X,Y)$ by 
	\[
	\lin{u_1}{u_2}{I(X,Y)} := u_1u_2\inv
	\]
	for $u_1,u_2 \in I(X,Y)$.
	We define a right action of $I(X)$ on $I(X,Y)$ by the composition from the right side and a right pairing on $I(X,Y)$ by 
	\[
	\rin{ u_1 }{ u_2 }{I(X,Y)} := u_1\inv u_2
	\]
	$u_1,u_2 \in I(X,Y)$.
	We can check that the set $I(X,Y)$ becomes a partial Morita equivalence from $I(Y)$ to $I(X)$ with respect to the above structures.
\end{exam}

\section{Adjointable maps}
\label{Section: Adjointable maps between inverse sets}

\subsection{Adjointable maps}
In this subsection, we introduce \emph{adjointable maps} on regular sets.
This notion is an analogy of the adjointable operators on right Hilbert $C^*$-modules (see \cite[p.8]{Lan95}).
From now on, we mainly consider right regular sets and right inverse sets with the right Hilbert $C^*$-module theory in mind.
We can give a similar theory for left regular sets and left inverse sets.
Let $T$ be an inverse semigroup and $\U$, $\V$, $\W$ be right regular $T$-sets.

\begin{defi}\label{Definition: adjointable}
	A map $\varphi\: \U \rightarrow \V$ is said to be \emph{adjointable} if there exists a map $\psi\: \V \rightarrow \U$ such that 
	\[
	\rin{ \psi(v) }{ u }{\U} = \rin{ v }{ \varphi(u) }{\V}
	\]
	holds for every $u \in \U$ and $v \in \V$.
	Such a map $\psi$ is said to be an \emph{adjoint} of $\varphi$.
	We denote the set of all adjointable maps from $\U$ to $\V$ as $\adL(\U, \V)$.
	We abbreviate $\adL(\U, \U)$ as $\adL(\U)$.
\end{defi}

\begin{lemm}\label{Lemma: adjointable composition}
	For every $\varphi_1\in\adL(\U,\V)$ and $\varphi_2\in \adL(\V,\W)$, $\varphi_2\varphi_1\in\adL(\U,\W)$ holds.
\end{lemm}
\begin{proof}
	Let $\psi_i$ be an adjoint of $\varphi_i$ for $i = 1,2$.
	For every $u \in \U$ and $w \in \W$,
	\begin{align*}
	\rin{ w }{ \varphi_2\varphi_1(u) }{\W} = \rin{ \psi_2(w) }{ \varphi_1(u) }{\V} = \rin{ \psi_1\psi_2(w) }{ u }{\U}
	\end{align*}
	holds.
	Thus $\psi_1\psi_2$ is an adjoint of $\varphi_2\varphi_1$.
\end{proof}
\begin{coro}
	The set $\adL(\U)$ becomes a semigroup with respect to the composition of maps.
\end{coro}

The next definition is an analogy of the rank-one operator on a right Hilbert $C^*$-module (see \cite[p.9]{Lan95}).
\begin{defi}
	For $u\in \U$, $v \in \V$, we define a map $\omega_{v,u} \: \U \rightarrow \V$ as 
	\[
	\omega_{v,u}(u') := v\rin{u}{u'}{\U}
	\]
	for every $u' \in U$.
	We set $\cpK(\U,\V) := \{ \omega_{v,u} \mid u\in \U, v \in \V \}$.
	We abbreviate $\cpK(\U,\U) $ as $\cpK(\U)$.
\end{defi}

\begin{lemm}\label{Lemma: omega_{vt,u}}
	For $t \in T$, $u\in \U$, and $v \in \V$, $\omega_{vt,u} = \omega_{v,ut^*}$ holds.
\end{lemm}
\begin{proof}
	For $t \in T$, $u,u'\in \U$, and $v \in \V$, 
	\[
	\omega_{vt,u}(u') = vt \rin{ u }{ u' }{\U} = v \rin{ ut^* }{ u' }{\U} = \omega_{ v, ut^*}(u')
	\]
	holds. 
	Thus $\omega_{vt,u} = \omega_{v,ut^*}$ holds.
\end{proof}

\begin{lemm}\label{Lemma: omega is adjointable}
	For every $u \in \U$ and $v\in \V$, the map $\omega_{u,v}\:\V\rightarrow\U$ is an adjoint of $\omega_{v,u}\:\U\rightarrow\V$.
\end{lemm}
\begin{proof}
	For every $u,u'\in \U$, $v,v'\in \V$,
	\begin{align*}
		\rin{ v' }{ \omega_{v,u} u' }{\V} 
		&= \rin{ v' }{ v \rin{u}{u'}{U} }{\V} \\
		&= \rin{ v' }{ v }{\V}\rin{ u }{ u' }{\U} \\
		&= \rin{ u \rin{v}{v'}{\V} }{ u' }{\U} \\
		&= \rin{ \omega_{u,v}(v') }{ u' }{\U}
	\end{align*}
	holds.
\end{proof}

\begin{lemm}\label{Lemma: omega_{w,v} phi}
	Let $\varphi\: \V \rightarrow \W$ be a right $T$-map.
	For every $u \in \U$ and  $v \in \V$, $\varphi\omega_{v,u} = \omega_{\varphi(v),u}$ holds.
\end{lemm}
\begin{proof}
	For every $u' \in \U$, we get
	\[
	\varphi\omega_{v,u}(u') = \varphi\big(v\rin{ u }{ u' }{\U}\big) = \varphi(v)\rin{ u }{ u' }{\U} = \omega_{\varphi(v),u}(u'). \qedhere
	\]
\end{proof}

\begin{lemm}\label{Lemma: phi omega_{v,u}}
	Let $\varphi\:\U\rightarrow\V$ be an adjointable map, and $\psi\:\V\rightarrow\U$ be an adjoint of $\varphi$.
	For every $v \in \V$ and $w \in \W$, $\omega_{w,v}\varphi = \omega_{w,\psi(v)}$ holds.
\end{lemm}
\begin{proof}
	For every $u \in \U$, we get
	\[
	\omega_{w,v}\varphi(u) = w \rin{v}{\varphi(u)}{\V} = w \rin{\psi(v)}{u}{\U} = \omega_{w,\psi(v)}(u). \qedhere
	\]
\end{proof}

\begin{coro}\label{Corollary: K(U) ideal}
	The set $\cpK(\U)$ is a two-sided ideal of the semigroup $\adL(\U)$.
\end{coro}
\begin{proof}
	This follow from Lemma \ref{Lemma: omega is adjointable}, \ref{Lemma: omega_{w,v} phi} and \ref{Lemma: phi omega_{v,u}}.
\end{proof}

\begin{lemm}\label{Lemma: generalized inverses of compact operators}
	For every $u \in \U$ and $v \in \V$, $\omega_{v,u}\,\omega_{u,v}\,\omega_{v,u} = \omega_{v,u}$ holds.
\end{lemm}
\begin{proof}
	For every $u \in \U$ and $v \in \V$, we get 
	\begin{align*}
	\omega_{v,u}\,\omega_{u,v}\,\omega_{v,u} 
	&= \omega_{ v \rin{u}{u}{\U} \rin{v}{v}{\V} , u }\\
	&= \omega_{ v \rin{v}{v}{\V} \rin{u}{u}{\U} , u }\\
	&= \omega_{ v \rin{v}{v}{\V} , u \rin{u}{u}{\U} }\\
	&= \omega_{v,u}.
	\end{align*}
	The third equal follows from Lemma \ref{Lemma: omega_{vt,u}}.
\end{proof}

\begin{coro}\label{Corollary: K(U) is regular}
	The subsemigroup $\cpK(\U)$ of $\adL(\U)$ is regular.
\end{coro}
\begin{proof}
	By Lemma \ref{Lemma: generalized inverses of compact operators}, an adjoint $\omega_{v,u}$ of $\omega_{u,v}$ is a generalized inverse of $\omega_{u,v}$ for every $u,v \in \U$.
\end{proof}

The following proposition is an analogy of Theorem \ref{Theorem: regular and inverse}.

\begin{prop}\label{Proposition: R-iv}
	For a right regular $T$-set $\U$, the following are equivalent: 
	\begin{enumerate}[(i)]
		\item $u \rin{ u' }{ u }{\U} = u$ and $u' \rin{ u }{ u' }{\U} = u'$ imply $u = u'$ for every $u,u' \in \U$ (that is, $\U$ is inverse),
		\item $\rin{ u }{ u }{\U} = \rin{u'}{u'}{\U} = \rin{ u }{ u' }{\U}$ implies $u = u'$ for every $u,u' \in \U$,
		\item $u \rin{ u }{ u' }{\U} = u' \rin{ u' }{ u }{\U} \rin{ u }{ u' }{\U}$ for every $u,u' \in \U$,
		\item $\omega_{u,u}$ and $\omega_{u',u'}$ commutes for every $u,u' \in \U$,
	\end{enumerate}
\end{prop}
\begin{proof}
	(i) $\Rightarrow$ (ii) is proved in Lemma \ref{Lemma: non-degenerate}.
	\begin{enumerate}
		\item[(ii) $\Rightarrow$ (iii)] 
		We set $x := u \rin{ u }{ u' }{\U}$ and $y := u' \rin{ u' }{ u }{\U} \rin{ u }{ u' }{\U}$. 
		We can easily check that $\rin{ x }{ x }{\U} = \rin{ y }{ y }{\U} = \rin{x}{y}{\U} = \rin{ u' }{ u }{\U}\rin{ u }{ u' }{\U}$ holds.
		(ii) implies $x = y$ holds.
		
		\item[(iii) $\Rightarrow$ (iv)] 
		By (iii) and Lemma \ref{Lemma: omega_{w,v} phi}, we have
		\begin{align*}
		\omega_{u,u} \omega_{u',u'} 
		&= \omega_{u \rin{ u }{ u' }{\U}, u'} \\
		&= \omega_{u' \rin{ u' }{ u }{\U} \rin{ u }{ u' }{\U}, u'}\\
		&= \omega_{u' \rin{ u' }{ u }{\U}, u' \rin{ u' }{ u }{\U}}\\
		&= \omega_{u' \rin{ u' }{ u }{\U}, u \rin{ u }{ u' }{\U} \rin{ u' }{ u }{\U}}\\
		&= \omega_{u' \rin{ u' }{ u }{\U} \rin{ u }{ u' }{\U} \rin{ u' }{ u }{\U}, u}\\
		&= \omega_{u' \rin{ u' }{ u }{\U}, u}\\
		&= \omega_{u', u'}\,\omega_{u, u}.
		\end{align*}
		
		\item[(iv) $\Rightarrow$ (i)] 
		Fix elements $u, u' \in \U$ such that $u \rin{ u' }{ u }{\U} = u$ and $u' \rin{ u }{ u' }{\U} = u'$.
		We get $\rin{ u }{ u' }{\U} \in E(S)$ by the same way as in Lemma \ref{Lemma: partial Morita context is a left and right inverse set}.
		Hence $\rin{ u }{ u' }{\U} = \rin{ u' }{ u }{\U}$.
		This implies that 
		\begin{align*}
		u = u \rin{ u }{ u' }{\U} = \omega_{u,u}(u') \text{ and } u' = u' \rin{ u' }{ u }{\U} = \omega_{u',u'}(u).
		\end{align*}
		Thus we have
		\begin{align*}
		u 
		&=  \omega_{u,u}(u')\\
		&= \omega_{u,u}\,\omega_{u',u'}(u')\\
		&= \omega_{u',u'}\,\omega_{u,u}(u')\\
		&= \omega_{u',u'}(u)\\
		&= u'.\qedhere
		\end{align*}
	\end{enumerate}
\end{proof}

From now on, we assume that $\U$, $\V$ and $\W$ are right inverse $T$-sets.

\begin{lemm}
	For $\varphi \in \adL(\U,\V)$, an adjoint of $\varphi$ is unique.
\end{lemm}
\begin{proof}
	Let $\psi_1,\psi_2\: \V \rightarrow \U$ be adjoints of $\varphi\: \U \rightarrow \V$.
	For every $u \in \U$ and $v \in \V$,
	\[
	\rin{ \psi_1(v) }{ u }{\U} = \rin{ v }{ \varphi(u) }{\V} = \rin{ \psi_2(v) }{ u }{\U}
	\]
	holds.
	By Lemma \ref{Lemma: non-degenerate} (ii), we get $\psi_1(v) = \psi_2(v)$ for every $v \in \V$.
	Thus we have $\psi_1 = \psi_2$.
\end{proof}

\begin{defi}
	We denote the adjoint of $\varphi\: \U \rightarrow \V$ as $\varphi\d \: \V \rightarrow \U$.
\end{defi}

\begin{lemm}\label{Lemma: adj T-map}
	For every $\varphi\in \adL(\U,\V)$, $\varphi$ is a right $T$-map.
\end{lemm}
\begin{proof}
	For every $u\in \U, v \in \V$ and $t \in T$,
	\begin{equation*}
	\rin{v}{\varphi(ut)}{\V}
	= \rin{\varphi\d(v)}{ut}{\U}
	= \rin{\varphi\d(v)}{u}{\U} t
	= \rin{v}{\varphi(u)}{\V} t
	= \rin{v}{\varphi(u)t}{\V}
	\end{equation*}
	holds. 
	By Lemma \ref{Lemma: non-degenerate} (ii), $\varphi(ut) = \varphi(u)t$ holds.
\end{proof}

\begin{prop}\label{Proposition: (psi phi)d = phid psid}
	For every $\varphi\in\adL(\U,\V)$ and $\psi\in \adL(\V,\W)$, $(\psi\varphi)\d = \varphi\d\psi\d$ holds.
\end{prop}
\begin{proof}
	See the proof of Lemma \ref{Lemma: adjointable composition}.
\end{proof}

\begin{prop}\label{Proposition: phidd = phi}
	For every $\varphi\in \adL(\U, \V)$, we have $\varphi^{\dag\dag} = \varphi$.
\end{prop}
\begin{proof}
	For every $u \in \U$ and $v \in \V$,
	\[
	\rin{ \varphi^{\dagger\dagger}(u) }{ v }{\V} = \rin{ u }{ \varphi\d(v) }{\U} = \rin{ \varphi(u) }{ v }{\V}
	\]
	holds, where the second equal follows by taking generalized inverses of the two sides of the equation in Definition \ref{Definition: adjointable}.
	By Lemma \ref{Lemma: non-degenerate} (ii), we get $\varphi^{\dagger\dagger}(u) = \varphi(u)$ for every $u \in \U$.
	Thus $\varphi^{\dagger\dagger} = \varphi$ holds.
\end{proof}

We will see that
\begin{itemize}
	\item $\cpK(\U)$ is an inverse subsemigroup of $\adL(\U)$ in Theorem \ref{Theorem: K(U) is inverse},
	\item $\adL(\U)$ is an inverse semigroup in Theorem \ref{Theorem: L(U) is inverse}, and
	\item $\adL(\U,\V)$ is a partial Morita equivalence from $\adL(\V)$ to $\adL(\U)$ in Theorem \ref{Theorem: L(U,V) inverse set}.
\end{itemize}

We first prove that $\cpK(\U)$ is an inverse semigroup.
We can easily see that $\omega_{u,u}$ is an idempotent of $\cpK(\U)$ for $u \in \U$.

\begin{lemm}\label{Lemma: E(K(U))}
	All idempotents of $\cpK(\U)$ is in the form of $\omega_{u, u}$ for some $u \in \U$.
\end{lemm}
\begin{proof}
	Fix $u,u' \in \U$ such that $\omega_{u,u'} \in E(\cpK(\U))$ holds.
	We put an idempotent $e$ of $T$ as $e := \rin{ u }{ u }{\U} \rin{ u' }{ u' }{\U}$.
	We get 
	\begin{align*}
	\rin{ ue }{ ue }{\U} &= \rin{ u'e }{ u'e }{\U} = e \text{ and }\\
	\rin{ u'e }{ ue }{\U}
	&= \rin{ u' \rin{ u }{ u }{\U} }{ u\rin{ u' }{ u' }{\U} }{\U}\\
	&= \rin{ u }{ u }{\U} \rin{ u' }{ u\rin{ u' }{ u' }{\U} }{\U}\\
	&= \rin{ u }{ u \rin{ u' }{ u\rin{ u' }{ u' }{\U} }{\U} }{\U} \\
	&= \rin{ u }{ \omega_{u, u'}\,\omega_{u, u'} u' }{\U}\\
	&= \rin{ u }{ \omega_{u, u'} u' }{\U}\\
	&= \rin{ u }{ u \rin{ u' }{ u' }{\U} }{\U} \\
	&= \rin{ u }{ u }{\U} \rin{ u' }{ u' }{\U}\\
	& = e.
	\end{align*}
	By Lemma \ref{Lemma: non-degenerate} (i), $ue = u'e$ holds.
	Thus we get
	\begin{align*}
	\omega_{u,u'} 
	&= \omega_{u \rin{ u }{ u }{\U}, u'\rin{ u' }{ u' }{\U}} \\
	&= \omega_{u \rin{ u }{ u }{\U}\rin{ u' }{ u' }{\U}, u' \rin{ u }{ u }{\U} \rin{ u' }{ u' }{\U}} \\
	&= \omega_{ue, u'e}. \qedhere
	\end{align*}
\end{proof}
 
\begin{theo}\label{Theorem: K(U) is inverse}
	For a right inverse $T$-set $\U$, the semigroup $\cpK(\U)$ is inverse.
	The generalize inverse $k^*$ of $k \in \cpK(\U)$ is its adjoint $k\d$.
\end{theo}
\begin{proof}
	By Corollary \ref{Corollary: K(U) is regular}, Lemma \ref{Lemma: E(K(U))}, and Proposition \ref{Proposition: R-iv}, $\cpK(\U)$ becomes a regular semigroup whose idempotents commute. 
	Theorem \ref{Theorem: regular and inverse} implies that $\cpK(\U)$ is an inverse semigroup.
\end{proof}

\begin{prop}
	Every right inverse $T$-set $\U$ becomes a Morita equivalence from $\cpK(\U)$ to $\rin{\U}{\U}{\U}$ with respect to the left action and the left pairing defined as $ku := k(u)$ and $\lin{ u }{ u' }{\U} := \omega_{ u, u' }$ for $k \in \cpK(\U)$, $u,u' \in \U$ respectively.
\end{prop}
\begin{proof}
	These structures satisfy (L-i)-(L-iii) by Lemma \ref{Lemma: omega_{w,v} phi}, \ref{Lemma: omega is adjointable}, and (R-iii).
	It is clear that the left and right pairings are compatible.
\end{proof}

We gave an analogy of Lemma \ref{Lemma: order}:
\begin{prop}\label{Proposition: order}
	For every $u,u' \in \U$, the following are equivalent;
	\begin{enumerate}[(i)]
		\item $u = u'\rin{ u }{ u }{\U}$,
		\item there exists $e \in E(T)$ with $u = u'e$,
		\item $u = \omega_{u,u}u'$,
		\item there exists $k \in E(\cpK(\U))$ with $u = k(u')$.
	\end{enumerate}
\end{prop}
\begin{proof}
	We can see that (i) implies (ii), and that (iii) implies (iv).
	If we assume (ii),
	\begin{align*}
	\omega_{u,u}u' &= u \rin{ u }{ u' }{\U} = u'e \rin{ u'e }{ u' }{\U} \\
	&= u'e \rin{ u' }{ u' }{\U} = u' \rin{ u' }{ u' }{\U} e = u'e = u.
	\end{align*}
	holds. 
	Thus we obtain (iii).
	If we assume (iv),
	\begin{align*}
	u'\rin{ u }{ u }{\U} &= u' \rin{ k(u') }{ k(u') }{\U} = u' \rin{ u' }{ k(u') }{\U} \\
	&= \omega_{u',u'}k(u') = k\omega_{u',u'}(u') = k(u') = u
	\end{align*}
	holds, where the second equal holds since we have $k\d = k^*$ by Theorem $\ref{Theorem: K(U) is inverse}$ and $k^* = k$ by $k\in E(\cpK(\U))$. 
	Thus we obtain (i).
\end{proof}

\begin{defi}\label{Definition: order}
	For $u, u' \in \U$, $u\leq u'$ if and only if $u$ and $u'$ satisfies one (and hence all) of the conditions in Proposition \ref{Proposition: order}.
\end{defi}
We can check that this binary relation $\leq$ on $\U$ becomes a partial order on $\U$.
Steinberg introduced order on Morita equivalences in \cite[Proposition 3.2, 3.5]{Ste11}.
For a right inverse $T$-set $\U$, Steinberg's order defined on the Morita equivalence $\U$ from $\cpK(\U)$ to $\rin{\U}{\U}{\U}$ coincides with our order on $\U$ in Definition \ref{Definition: order}.
We can see easily the following:
For every $k \in \cpK(\U)$, $u,u'\in \U$, and $t \in T$, $u\leq u'$ implies $ut \leq u't$ and $k(u) \leq k(u')$.
Let $u_i,u_i'\in \U$ with $u_i \leq u_i'$ for $i = 1,2$.
We obtain $\rin{ u_1 }{ u_2 }{\U} \leq \rin{ u_1' }{ u_2' }{\U}$ and $\omega_{u_1,u_2} \leq \omega_{u_1',u_2'}$.

We proceed to show that $\adL(\U)$ is an inverse semigroup and $\adL(\U,\V)$ is a right inverse $\adL(\U)$-set.
By virtue of Corollary \ref{Corollary: K(U) ideal} and Theorem \ref{Theorem: K(U) is inverse}, we have the following key lemma:
\begin{lemm}\label{Lemma: phi and omega commute}
	For every idempotent $\varphi$ of $\adL(\U)$ and $u \in \U$, $\varphi$ and $\omega_{u,u}$ commute.
\end{lemm}
\begin{proof}
	Apply Theorem \ref{Proposition: e and f commute} to $S = \adL(\U)$ and $I = \cpK(\U)$.
\end{proof}

\begin{lemm}\label{Lemma: phi omega_{u,u} idempotent}
	For every $\varphi\in \adL(\U,\V)$ and $u \in \U$, $\varphi\d\varphi\,\omega_{u,u}$ and $\omega_{u,u}\,\varphi\d\varphi$ are idempotents of $\adL(\U)$.
\end{lemm}
\begin{proof}
	By Lemma \ref{Lemma: omega_{w,v} phi} (i), (ii), and Corollary \ref{Proposition: phidd = phi},
	\begin{align*}
	\varphi\d\varphi\,\omega_{u,u}\,\varphi\d\varphi\,\omega_{u,u} 
	&= \varphi\d\,\omega_{\varphi(u),\varphi(u)}\,\omega_{\varphi(u),u}\\
	&= \varphi\d\,\omega_{\varphi(u)\rin{ \varphi(u) }{ \varphi(u) }{\V}, u}\\
	&= \varphi\d\,\omega_{\varphi(u),u}\\
	&= \varphi\d\varphi\,\omega_{u,u}
	\end{align*}
	holds for every $\varphi\in \adL(\U,\V)$ and $u \in \U$.
	In a similar way, $\omega_{u,u}\,\varphi\d\varphi$ is an idempotent for every $\varphi\in \adL(\U,\V)$ and $u \in \U$.
\end{proof}

\begin{lemm}\label{Lemma: phid phi and omega commute}
	For every $\varphi\in \adL(\U,\V)$ and $u \in \U$, $\varphi\d\varphi$ and $\omega_{u,u}$ commute.
\end{lemm}
\begin{proof}
	For every $\varphi\in \adL(\U,\V)$ and $u \in \U$, $\varphi\d\varphi\omega_{u,u}$ and $\omega_{u,u}\varphi\d\varphi$ commute with $\omega_{u,u}$ by Lemma \ref{Lemma: phi and omega commute} and \ref{Lemma: phi omega_{u,u} idempotent}.
	Thus we get 
	\begin{align*}
	\varphi\d\varphi\,\omega_{u,u} 
	&= \varphi\d\varphi \, (\omega_{u,u} \, \omega_{u,u})\\
	&= (\varphi\d\varphi \, \omega_{u,u}) \, \omega_{u,u}\\
	&= \omega_{u,u} \, (\varphi\d\varphi \, \omega_{u,u})\\
	&= (\omega_{u,u} \, \varphi\d\varphi) \, \omega_{u,u}\\
	&= \omega_{u,u} \, (\omega_{u,u} \, \varphi\d\varphi)\\
	&= (\omega_{u,u} \, \omega_{u,u}) \, \varphi\d\varphi\\
	&= \omega_{u,u} \, \varphi\d\varphi.\qedhere
	\end{align*}
\end{proof}

\begin{prop}\label{Proposition: phi phid phi = phi}
	For every $\varphi\in \adL(\U,\V)$, $\varphi\varphi\d\varphi = \varphi$ and $\varphi\d\varphi\varphi\d = \varphi\d$ hold.
\end{prop}
\begin{proof}
	For every $\varphi \in \adL(\U)$ and $u \in \U$, we get
	\begin{align*}
	\varphi(u) 
	&= \varphi(u)\rin{ \varphi(u) }{ \varphi(u) }{\V}\\
	&= \varphi\big(u\rin{ \varphi(u) }{ \varphi(u) }{\V}\big)\\
	&= \varphi\left( u \rin{ u }{ \varphi\d\varphi(u) }{\U} \right)\\
	&= \varphi\big( \omega_{u,u}\,\varphi\d\varphi(u) \big)\\
	&= \varphi\big( \varphi\d\varphi\,\omega_{u,u}(u) \big)\\
	&= \varphi\varphi\d\varphi(u),
	\end{align*}
	where the second equal follows from Lemma \ref{Lemma: adj T-map}, and the fifth equal follows from Lemma \ref{Lemma: phid phi and omega commute}.
	Thus $\varphi\varphi\d\varphi = \varphi$ holds.
	This implies that $\varphi\d = (\varphi\varphi\d\varphi)\d = \varphi\d\varphi\varphi\d$.
\end{proof}

\begin{lemm}\label{Lemma: u phi(u) u}
	For every idempotent $\varphi$ of $\adL(\U)$ and $u \in \U$, the following hold:
	\begin{enumerate}[(i)]
		\item $\varphi(u) = u \rin{ u }{ \varphi(u) }{\U}$.
		\item $\rin{ u }{ \varphi(u) }{\U} = \rin{ \varphi(u) }{ u }{\U} = \rin{ \varphi(u) }{ \varphi(u) }{\U}$.
	\end{enumerate}
	By (i) and (ii), $\varphi(u) = u \rin{ u }{ \varphi(u) }{\U} = u \rin{ \varphi(u) }{ u }{\U} = u \rin{ \varphi(u) }{ \varphi(u) }{\U}$ holds.
\end{lemm}
\begin{proof}
	For every idempotent $\varphi$ of $\adL(\U)$ and $u \in \U$, Lemma \ref{Lemma: phi and omega commute} implies (i) as 
	\[
	\varphi(u) = \varphi\left( \omega_{u,u}(u) \right) = \omega_{u,u}(\varphi(u)) = u \rin{ u }{ \varphi(u) }{\U}.
	\]
	Lemma \ref{Lemma: phi and omega commute} also implies that
	\begin{align*}
	\varphi(u) 
	&= \varphi(u) \rin{ \varphi(u) }{ \varphi(u) }{\U}\\
	&= \omega_{\varphi(u),\varphi(u)}(\varphi(u))\\
	&= \varphi\left( \omega_{\varphi(u),\varphi(u)}(u) \right)\\
	&= \varphi( \varphi(u) \rin{ \varphi(u) }{ u }{\U} )\\
	&= \varphi( \varphi(u) ) \rin{ \varphi(u) }{ u }{\U}\\
	&= \varphi(u) \rin{ \varphi(u) }{ u }{\U}.
	\end{align*}
	This implies that 
	\begin{align*}
	\rin{ \varphi(u) }{ \varphi(u) }{\U} 
	&= \rin{ \varphi(u) }{ \varphi(u) \rin{ \varphi(u) }{ u }{\U} }{\U}\\
	&= \rin{ \varphi(u) }{ \varphi(u) }{\U} \rin{ \varphi(u) }{ u }{\U}\\
	&= \rin{ \varphi(u) }{ u }{\U}
	\end{align*}
	holds.
	Thus we have (ii).
\end{proof}

\begin{lemm}\label{Lemma: phi d = phi}
	For every idempotent $\varphi$ of $\adL(\U)$, $\varphi\d = \varphi$ holds.
\end{lemm}
\begin{proof}
	For every $u,u' \in \U$ and an idempotent $\varphi$ of $\adL(\U)$, we get
	\begin{align*}
	\rin{ u' }{ \varphi(u) }{\U}
	&= \rin{ u' }{ u\rin{ \varphi(u) }{ u }{\U}}{\U}\\
	&= \rin{ u' }{ u }{\U}\rin{ \varphi(u) }{ u }{\U}\\
	&= \rin{ \varphi( u ) \rin{ u }{ u' }{\U}  }{ u }{\U}\\
	&= \rin{ \varphi( u \rin{ u }{ u' }{\U})  }{ u }{\U}\\
	&= \rin{ \varphi(\omega_{u,u} u' ) }{ u }{\U}\\
	&= \rin{ \omega_{u,u}\varphi( u' ) }{ u }{\U}\\
	&= \rin{ \varphi( u' ) }{ \omega_{u,u}u }{\U}\\
	&= \rin{ \varphi(u') }{ u }{\U},
	\end{align*}
	where the first equal follows from Lemma \ref{Lemma: u phi(u) u}, the third equal follows from Lemma \ref{Lemma: adj T-map}, the fifth equal follows from Lemma \ref{Lemma: phi and omega commute}, and the sixth equal follows from Lemma \ref{Lemma: omega is adjointable}.
\end{proof}

\begin{prop}\label{Proposition: phi psi phi = phi}
	Let $\varphi$ be an element of $\adL(\U,\V)$.
	If $\psi_1,\psi_2 \in \adL(\V,\U)$ satisfies $\varphi\psi_i\varphi = \varphi$ and $\psi_i\varphi\psi_i = \psi_i$ for $i = 1,2$, then $\psi_1 = \psi_2$ holds.
\end{prop}
\begin{proof}
	Notice that $\varphi\psi_i$ is an idempotent of $\adL(\V)$ and $\psi_i\varphi$ is an idempotent of $\adL(\U)$ for $i = 1,2$.
	By Lemma \ref{Lemma: phi d = phi}, we get $(\varphi\psi_i)\d = \varphi\psi_i $ and $(\psi_i\varphi)\d = \psi_i\varphi$ for $ i = 1,2$.
	Thus we have 
	\begin{align*}
	\psi_1 &= \psi_1\varphi\psi_1
	\\
	&= (\psi_1\varphi)\d\psi_1
	\\
	&= (\psi_1\varphi\psi_2\varphi)\d\psi_1
	\\
	&= (\psi_2\varphi)\d(\psi_1\varphi)\d\psi_1
	\\
	&= \psi_2\varphi\psi_1\varphi\psi_1
	\\
	&= \psi_2\varphi\psi_1
	\end{align*}
	and 
	\begin{align*}
	\psi_2 &= \psi_2\varphi\psi_2\\
	&= \psi_2(\varphi\psi_2)\d\\
	&= \psi_2(\varphi\psi_1\varphi\psi_2)\d\\
	&= \psi_2(\varphi\psi_2)\d(\varphi\psi_1)\d\\
	&= \psi_2\varphi\psi_2\varphi\psi_1\\
	&= \psi_2\varphi\psi_1.
	\end{align*}
	Hence $\psi_1 = \psi_2$ holds.
\end{proof}

Now we obtain our desired theorems:

\begin{theo}\label{Theorem: L(U) is inverse}
	For a right inverse $T$-set $\U$, the semigroup $\adL(\U)$ is inverse.
\end{theo}
\begin{proof}
	This follows from Proposition \ref{Proposition: phi phid phi = phi} and \ref{Proposition: phi psi phi = phi}.
\end{proof}

\begin{theo}\label{Theorem: L(U,V) inverse set}
	Let $\U$ and $\V$ be right inverse $T$-sets.
	The set $\adL(\U,\V)$ becomes a partial Morita equivalence from $\adL(\V)$ to $\adL(\U)$ with respect to the following structures: 
	The left and right actions are defined by composing from the left side and the right side respectively.
	The left and right pairings are defined as $\lin{ \varphi_1 }{ \varphi_2 }{\adL(\U,\V)}:= \varphi_1\varphi_2\d$ and $\rin{ \varphi_1 }{ \varphi_2 }{\adL(\U,\V)}:= \varphi_1\d\varphi_2$ for every $\varphi_1,\varphi_2\in\adL(\U,\V)$ respectively.
	With the same structures, $\cpK(\U,\V)$ becomes a partial Morita equivalence from $\cpK(\V)$ to $\cpK(\U)$.
\end{theo}
\begin{proof}	
	The associative law of composition of maps implies that $\adL(\U,\V)$ becomes a right $\adL(\U)$-set and $\rin{ \cdot }{ \cdot }{\adL(\U,\V)}$ satisfies the condition (R-i) in Definition \ref{Definition: inverse set}.
	Proposition \ref{Proposition: (psi phi)d = phid psid} implies (R-ii).
	Proposition \ref{Proposition: phi phid phi = phi} implies (R-iii).
	Corollary \ref{Proposition: phidd = phi} and Proposition \ref{Proposition: phi psi phi = phi} imply (R-iv).
	We can also check that $\cpK(\U,\V)$ becomes a partial Morita equivalence from $\cpK(\V)$ to $\cpK(\U)$.
\end{proof}

\section{Inverse correspondences and their tensor products}
\label{Section: Inverse correspondences and their tensor products}

In this section, we introduce a notion of \emph{inverse correspondence} between inverse semigroups and their tensor product.
In the $C^*$-algebra theory, a \emph{\Ccorr} from a $C^*$-algebra $A$ to $B$ consists of a right Hilbert $B$-module $\E$ and a \shom from $A$ to the $C^*$-algebra $\adL(\E)$ of all adjointable maps on $\E$.
By virtue of Theorem \ref{Theorem: L(U) is inverse}, we can define inverse correspondences in a similar way to $C^*$-correspondences.
Let $S$ and $T$ be inverse semigroups.
\begin{defi}\label{Definition: inv corr}
	An \emph{inverse correspondence} $\U$ from $S$ to $T$, denoted as $\U\: S \rightarrow T$, is a right inverse $T$-set $\U$ equipped with a semigroup homomorphism $\theta_{\U}\: S\rightarrow\adL(\U)$.
\end{defi}

We denote $\theta_{\U}(s)(u)$ as $su$ for every $s \in S$ and $u \in \U$.

\begin{lemm}\label{Lemma: left action basics}
	Let $\U \: S \rightarrow T$ be an inverse correspondence.
	The following hold:
	\begin{enumerate}[(i)]
		\item $S \times \U \rightarrow \U; (s,u)\mapsto su$ is a left action of $S$ on $\U$.
		\item $s(ut) = (su)t$ for every $s \in S$, $u \in \U$ and $t \in T$.
		\item $\rin{ u' }{ su }{\U} = \rin{ s^*u' }{ u }{\U}$ for every $s \in S$, $u,u' \in \U$.
	\end{enumerate}
\end{lemm}
\begin{proof}
	(i) is clear since $\theta_\U$ is a semigroup homomorphism.
	(ii) follows from Lemma \ref{Lemma: adj T-map}.
	(iii) holds because $\theta_{\U}$ keeps generalized inverses.
\end{proof}

\begin{lemm}\label{Lemma: left action basics 1.5}
	A right inverse $T$-set $\U$ equipped with a left action of $S$ on $\U$ such that $\rin{ u' }{ su }{\U} = \rin{ s^*u' }{ u }{\U}$ holds for every $u,u' \in \U$, $s \in S$ becomes an inverse correspondence from $S$ to $T$.
\end{lemm}
\begin{proof}
	We can check easily that $\theta(s) \:\U \rightarrow \U; u\mapsto su$ is an adjointable map on $\U$ for every $s\in S$, and $\theta\: S \rightarrow \adL(\U);s \mapsto \theta(s)$ is a semigroup homomorphism.
\end{proof}

\begin{lemm}\label{Lemma: left action basics2}
	Let $\U \: S \rightarrow T$ be an inverse correspondence.
	For $e \in E(S)$, $\varphi\in E(\adL(\U))$ and $u \in \U$, $e\varphi(u) = \varphi(eu)$ holds.
\end{lemm}
\begin{proof}
	The semigroup homomorphism $\theta_{\U}$ keeps idempotents.
	Two idempotents $\theta_{\U}(e)$ and $\varphi$ in $\adL(\U)$ commutes by Theorem \ref{Theorem: L(U) is inverse}.
\end{proof}

For an inverse correspondence $\U$ from $S$ to $T$, 
\[
S\U := \{ su\in \U \mid s \in S, u \in \U \}
\]
becomes an inverse correspondence from $S$ to $T$.
\begin{defi}\label{Definition: non-deg corr}
	An inverse correspondence $\U\: S \rightarrow T$ is \emph{non-degenerate} if  $S\U = \U$.
\end{defi}

\begin{exam}
	A partial Morita equivalence $\U$ from $S$ to $T$ becomes an inverse correspondence by forgetting the left pairing.
	We can check this fact as follows:
	Let $\U$ be a partial Morita equivalence from $S$ to $T$. 
	For $s\in S, u,u'\in\U$, we have 
	\[
	u \rin{ su }{ u' }{\U} = \lin{ u }{ su }{\U} u' = \lin{ u }{ u }{\U} s^* u' = u \rin{ u }{ s^*u' }{\U},
	\]
	where the second equal follows from Lemma \ref{Lemma: left pairing basics}.
	Thus we obtain
	\begin{align*}
		\rin{ u' }{ su }{\U} 
		&= \rin{ u' }{ su\rin{ u }{ u }{\U} }{\U}\\
		&= \rin{ u' }{ su }{\U} \rin{ u }{ u }{\U}\\
		&= \rin{ u \rin{ su }{ u' }{\U} }{ u }{\U}\\
		&= \rin{ u \rin{ u }{ s^*u' }{\U} }{ u }{\U}\\
		&= \rin{ s^*u' }{ u }{\U} \rin{ u }{ u }{\U}\\
		&= \rin{ s^*u' }{ u \rin{ u }{ u }{\U} }{\U}\\
		&= \rin{ s^*u' }{ u }{\U},
	\end{align*}
	where the third and fifth equals follow from Lemma \ref{Lemma: left pairing basics} (this is an another proof of \cite[Proposition 2.3 (10)]{Ste11}).
	Thus $\U$ is an inverse correspondence from $S$ to $T$ by Lemma \ref{Lemma: left action basics 1.5}.
	For every element $u\in\U$, $u = \lin{u}{u}{\U}u \in S\U$ holds.
	Thus $\U$ is non-degenerate.
\end{exam}

\begin{lemm}\label{Lemma: partial Morita iff...}
	An inverse correspondence $\U\:S \rightarrow T$ comes from a partial Morita equivalence if and only if there exists a two-sided ideal $I$ of $S$ such that $\theta_{\U}|_I\: I \rightarrow \adL(\U)$ is an isomorphism onto $\cpK(\U)$.
\end{lemm}
\begin{proof}
	For a partial Morita equivalence $\U$ from $S$ to $T$, the subset $I = \lin{\U}{\U}{\U}$ is a two-sided ideal of $S$.
	By the compatibility of left and right pairings, we have $\theta_\U(\lin{u_1}{u_2}{\U}) = \omega_{u_1,u_2} \in \cpK(\U)$ for every $u_1,u_2\in\U$.
	Thus $\theta_\U|_I$ is a semigroup homomorphism onto $\cpK(\U)$.
	For $s_1,s_2\in \lin{\U}{\U}{\U}$ with $\theta_\U(s_1) = \theta_\U(s_2)$, we have 
	\[
	s_1\lin{u}{u'}{\U} = \lin{s_1u}{u'}{\U} = \lin{s_2u}{u'}{\U} = s_2\lin{u}{u'}{\U}
	\]
	for every $u,u'\in\U$.
	This implies that $s_1 = s_2$ by applying Lemma \ref{Lemma: right cancel} for the inverse semigroup $\lin{\U}{\U}{\U}$.
	Thus $\theta_\U|_I$ is an isomorphism onto $\cpK(\U)$.
	
	We assume that $\theta_\U|_I$ is an isomorphism onto $\cpK(\U)$ with some two-sided ideal $I$ of $S$.
	For every $u_1,u_2 \in \U$, we set $\lin{u_1}{u_2}{\U}:= (\theta_\U|_I)\inv\big(\omega_{ u_1, u_2 }\big) \in I \subset S$.
	We can check easily that $\U$ and $\lin{\cdot}{\cdot}{\U}$ form a regular left $S$-set.
	The compatibility of left and right pairings in Definition \ref{Definition: partial Morita context} is clear by the definition of the left pairing.
\end{proof}

The following lemma means that we can reconstruct a left pairing of a partial Morita equivalence $\U$ from the other structures of $\U$.
This fact is nothing but an analogy of \cite[Lemma 2.4]{Kat03}.

\begin{lemm}\label{Lemma: left pairing unique}
	If two partial Morita equivalence produces the same inverse correspondence by forgetting the left pairings, then they are same as partial Morita equivalences.
\end{lemm}
\begin{proof}
	It follows from Lemma \ref{Lemma: two-sided ideals coincide}.
\end{proof}

The following corollary follows from Lemma \ref{Lemma: partial Morita iff...}:
\begin{coro}\label{Corollary: partial Morita iff...}
	Let $\U$ be an inverse correspondence from $S$ to $T$.
	\begin{enumerate}[(i)]
		\item $\U$ comes from a partial Morita equivalence which is left full if and only if $\theta_{\U} \: S \rightarrow \adL(\U)$ is an isomorphism onto $\cpK(\U)$.
		\item$\U$ comes from a Morita equivalence if and only if $\theta_{\U} \: S \rightarrow \adL(\U)$ is an isomorphism onto $\cpK(\U)$ and $\U$ is right full.
	\end{enumerate}
\end{coro}

We give another example of non-degenerate inverse correspondence.
\begin{exam}
	For a semigroup homomorphism $\theta \: S \rightarrow T$, the subset $\U_\theta := \{ \theta(s)t \mid s \in S, t \in T \}$ of $T$ becomes a non-degenerate inverse correspondence from $S$ to $T$ with respect to the following structures:
	The right action of $T$ is defined as the multiplication from the right hand side.
	The right pairing is defined by $\rin{u_1}{u_2}{\U_\theta}: = u_1^*u_2 \in T$ for every $u_1,u_2\in\U_\theta$.
	The left action of $S$ is defined as $S\times \U_\theta \rightarrow \U_\theta; (s,u) \mapsto \theta(s)u$.
	Especially, taking $\theta$ as the identity homomorphism for $S$, we get a Morita equivalence $S$ from $S$ to $S$.
\end{exam}

We define tensor products of inverse correspondences.
This is a generalization of the tensor product of Morita contexts introduced in \cite[Proposition 2.5]{Ste11}.

Let $S_2$ and $S_3$ be inverse semigroups, $\U$ be a right inverse $S_2$-set, and $\V\: S_2 \rightarrow S_3$ be an inverse correspondence.
We define a set $\U \otimes \V$ as the quotient of $\U \times \V$ by the least equivalence relation $\sim$ such that $(us_2,v) \sim (u,s_2v)$ for all $u\in \U$, $v \in \V$ and $s_2 \in S_2$.
The equivalence class of $(u,v) \in \U\times\V$ is denoted as $u\otimes v \in \U\otimes\V$.
We define a right action of $S_3$ on $\U \otimes \V$ as 
\[
(u\otimes v)s_3 = u \otimes (vs_3),
\]
and a right pairing $(\U \otimes \V) \times (\U \otimes \V) \rightarrow S_3$ as 
\[
\rin{ u' \otimes v' }{ u \otimes v }{\U\otimes\V} := \rin{ v' }{ \rin{ u' }{ u }{\U} v }{\V}
\]
for all $u,u'\in \U$, $v,v' \in \V$, and $s_3 \in S_3$.

\begin{prop}
	The set $\U \otimes \V$ becomes an inverse $S_3$-set.
\end{prop}
\begin{proof}
	We can easily check the well-definedness of the right action of $S_3$ on $\U\otimes\V$.
	For every $u,u' \in \U$, $v,v'\in \V$, $s_2\in S_2$, we get 
	\begin{align*}
	\rin{ v' }{ \rin{u's_2}{u}{\U} v }{\V} &= \rin{ v' }{ s_2^*\rin{ u' }{ u }{\U} v }{\V} = \rin{ s_2 v' }{ \rin{ u' }{ u }{\U} v }{\V}, \text{ and }\\
	\rin{ v' }{ \rin{ u' }{ u s_2 }{\U} v }{\V} &= \rin{ v' }{ \rin{ u' }{ u }{\U} s_2 v }{\V}.
	\end{align*}
	Thus the map $\rin{ \cdot }{ \cdot }{\U\otimes\V}\: (\U \otimes \V) \times (\U \otimes \V) \rightarrow S_3$ is well-defined.
	
	We can easily check the condition (R-i) in Definition \ref{Definition: inverse set}.
	
	For every $u,u'\in\U$ and $v,v'\in\V$,
	\begin{align*}
	\rin{ u \otimes v }{ u' \otimes v' }{\U\otimes\V}^* 
	&= \rin{ v }{ \rin{ u }{ u' }{\U} v' }{\V}^* \\
	&= \rin{ \rin{ u }{ u' }{\U}^* v }{ v' }{\V}^* \\
	&= \rin{ \rin{ u' }{ u }{\U} v }{ v' }{\V}^* \\
	&= \rin{ v' }{ \rin{ u' }{ u }{\U} v }{\V} \\
	&= \rin{ u' \otimes v' }{ u \otimes v }{\U\otimes\V}
	\end{align*}
	holds, where the second equal follows from Lemma \ref{Lemma: left action basics} (iii).
	Thus we see the condition (R-ii) in Definition \ref{Definition: inverse set}.
	
	For every $u\in \U$ and $v\in \V$, we get 
	\begin{align*}
	(u\otimes v) \rin{ u\otimes v }{ u\otimes v }{\U\otimes\V}
	&= u \otimes v \rin{ v }{ \rin{ u }{ u }{\U} v }{\V}\\
	&= u \otimes v \rin{ \rin{ u }{ u }{\U} v }{ \rin{ u }{ u }{\U} v }{\V}\\
	&= u \rin{ u }{ u }{\U} \otimes v \rin{ \rin{ u }{ u }{\U} v }{ \rin{ u }{ u }{\U} v }{\V}\\
	&= u \otimes \rin{ u }{ u }{\U} v \rin{ \rin{ u }{ u }{\U} v }{ \rin{ u }{ u }{\U} v }{\V}\\
	&= u \otimes \rin{ u }{ u }{\U} v \\
	&= u \rin{ u }{ u }{\U} \otimes v \\
	&= u \otimes v,
	\end{align*}
	where the second equal follows from Lemma \ref{Lemma: left action basics} (iii).
	Thus the condition (R-iii) in Definition \ref{Definition: inverse set} holds.
	Hence $\U\otimes\V$ becomes a right regular $S_3$-set.
	
	We see that the condition (iii) in Proposition \ref{Proposition: R-iv} holds for checking that the condition (R-iv) holds.
	For every $u,u' \in \U$ and $v,v' \in \V$, we get 
	\begin{align*}
	&(u'\otimes v') \rin{ u'\otimes v' }{ u\otimes v }{\U\otimes\V}\rin{ u\otimes v }{ u'\otimes v' }{\U\otimes\V} \\
	&= u' \otimes v' \rin{ v' }{ \rin{ u' }{ u }{\U} v }{\V}\rin{ v }{ \rin{ u }{ u' }{\U} v' }{\V} \\
	&= u' \otimes v' \rin{ v' }{ \rin{ u' }{ u }{\U} v }{\V}\rin{ \rin{ u' }{ u }{\U} v }{ v' }{\V} \\
	&= u' \otimes \omega_{v',v'}\, \omega_{\rin{ u' }{ u }{\U} v, \rin{ u' }{ u }{\U} v} (v')\\
	&= u' \otimes \omega_{\rin{ u' }{ u }{\U} v, \rin{ u' }{ u }{\U} v} \, \omega_{v',v'} (v')\\
	&= u' \otimes \rin{ u' }{ u }{\U} v \rin{ \rin{ u' }{ u }{\U} v }{ v' }{\V} \\
	&= u' \otimes \rin{ u' }{ u }{\U} v \rin{ v }{ \rin{ u }{ u' }{\U} v' }{\V} \\
	&= u' \rin{ u' }{ u }{\U} \otimes v \rin{ v }{ \rin{ u }{ u' }{\U} v' }{\V} \\
	&= u \rin{ u }{ u' }{\U} \rin{ u' }{ u }{\U} \otimes v \rin{ v }{ \rin{ u }{ u' }{\U} v' }{\V} \\
	&= u \otimes \rin{ u }{ u' }{\U} \rin{ u' }{ u }{\U} v \rin{ v }{ \rin{ u }{ u' }{\U} v' }{\V} \\
	&= u \otimes \rin{ u }{ u' }{\U} \rin{ u' }{ u }{\U} \omega_{v,v}( \rin{ u }{ u' }{\U} v' ) \\
	&= u \otimes \omega_{v,v}\left( \rin{ u }{ u' }{\U} \rin{ u' }{ u }{\U}\rin{ u }{ u' }{\U} v' \right) \\
	&= u \otimes \omega_{v,v}\left( \rin{ u }{ u' }{\U} v' \right) \\
	&= u \otimes v \rin{ v }{ \rin{ u }{ u' }{\U} v' }{\V} \\
	&= (u \otimes v) \rin{ u \otimes v }{ u' \otimes v' }{\U\otimes\V},
	\end{align*}
	where the second and sixth equals follow from Lemma \ref{Lemma: left action basics}(iii), the eleventh equal follows from Lemma \ref{Lemma: left action basics2}, and the fourth and eighth equals follow from Proposition \ref{Proposition: R-iv}.
	Thus $\U\otimes\V$ becomes an inverse $S_3$-set.
\end{proof}

\begin{lemm}\label{Lemma: full}
	If $\U\otimes \V$ is right full, then so is $\V$.
\end{lemm}
\begin{proof}
	We can see
	\[
	S_3 = \rin{ \U\otimes\V }{ \U\otimes\V }{\U\otimes\V} = \rin{ V }{ \rin{ \U }{ \U }{\U}V }{\V} \subset \rin{ V }{ V }{\V} \subset S_3.\qedhere
	\]
\end{proof}

Let $S_1$, $S_2$, $S_3$ be inverse semigroups, and $\U\: S_1 \rightarrow S_2$, $\V\: S_2 \rightarrow S_3$ be inverse correspondences.
We define a left action of $S_1$ on the right inverse $S_3$-set $\U\otimes\V$ as 
\[
s_1(u\otimes v) = (s_1u) \otimes v
\]
for $u\in \U$, $v\in \V$, and $s_1 \in S_1$.

\begin{prop}\label{Proposition: tensor inverse corr}
	The right inverse $S_3$-set $\U\otimes\V$ becomes an inverse correspondence from $S_1$ to $S_3$ with respect to the left action of $S_1$ defined above.
	If $\U$ is non-degenerate, then so is $\U \otimes \V$.
\end{prop}
\begin{proof}
	We can easily check the well-definedness of the left action of $S_1$ on $\U\otimes\V$ and that $\U\otimes\V$ becomes a $S_1\hi S_3$ biset.
	For every $u,u' \in \U$, $v,v'\in \V$, $s_1\in S_1$, we get 
	\begin{align*}
	\rin{ u' \otimes v' }{ s_1(u \otimes v) }{\U\otimes\V}
	&= \rin{ u' \otimes v' }{ (s_1u) \otimes v }{\U\otimes\V}\\
	&= \rin{ v' }{ \rin{u'}{s_1u}{\U} v }{\V} \\
	&= \rin{ v' }{ \rin{s_1^*u'}{u}{\U} v }{\V}\\
	&= \rin{ (s_1^*u') \otimes v' }{ u \otimes v }{\U\otimes\V}\\
	&= \rin{ s_1^*(u' \otimes v') }{ u \otimes v }{\U\otimes\V}.
	\end{align*}
	By Lemma \ref{Lemma: left action basics 1.5}, $\U\otimes\V$ becomes an inverse correspondence from $S_1$ to $S_3$.
	
	Assume that $\U$ is non-degenerate.
	Take $u\in\U$ and $v\in\V$ arbitrarily.
	There exists $s' \in S$ and $u' \in \U$ with $u = s'u'$.
	Hence we get $u \otimes v = s'u' \otimes v = s'(u'\otimes v) \in S(\U\otimes\V)$.
	Thus $\U\otimes\V$ is non-degenerate.
\end{proof}

\begin{exam}\label{Example: tensor Morita equiv}
	If two inverse correspondences $\U\:S_1 \rightarrow S_2$ and $\V\:S_2\rightarrow S_3$ come from partial Morita equivalences, then so does the inverse correspondence $\U\otimes\V\: S_1 \rightarrow S_3$.
	We can check this fact as follows:
	By Lemma \ref{Lemma: partial Morita iff...}, there exist the two-sided ideals $I, J$ of $S_1$, $S_2$ such that $\theta_{\U}|_{I}$, $\theta_{\V}|_{J}$ are isomorphisms onto $\cpK(\U)$, $\cpK(\V)$ respectively.
	We define a subset $W$ of $S_1$ as 
	\[
	W := \left\{ s\in I \, \middle| \, \rin{ u' }{ \theta_\U(s)u }{\U}\in J \text{ for all } u,u'\in\U \right\}.
	\]
	It is easy to see that this is a two-sided ideal of $S_1$.
	We can prove the restriction $\theta_{\U\otimes\V}|_W$ of $\theta_{\U\otimes\V}\:S_1\rightarrow\adL(\U\otimes\V)$ is an isomorphism onto $\cpK(\U\otimes\V)$ as follows:
	Take $s,s'\in W$ with $\theta_{\U\otimes\V}(s) = \theta_{\U\otimes\V}(s')$.
	For $u,u'\in\U$ and $v,v'\in\V$, we have 
	\begin{align*}
	\rin{ v' }{ \theta_\V\big(\rin{ u' }{ \theta_\U(s)(u) }{\U}\big)(v) }{\V} 
	&= \rin{ u' \otimes v' }{ \theta_{\U\otimes\V}(s)(u\otimes v) }{\U\otimes\V} \\
	&= \rin{ u' \otimes v' }{ \theta_{\U\otimes\V}(s')(u\otimes v) }{\U\otimes\V} \\
	&= \rin{ v' }{ \theta_\V\big(\rin{ u' }{ \theta_\U(s')(u) }{\U}\big)(v) }{\V}.
	\end{align*}
	By Lemma \ref{Lemma: non-degenerate} and the injectivity of $\theta_\V|_J$, we have 
	\[
	\rin{ u' }{ \theta_\U(s)(u) }{\U} = \rin{ u' }{ \theta_\U(s')(u) }{\U}.
	\]
	In a similar way, we obtain $s = s'$.
	Thus $\theta_{\U\otimes\V}$ is injective on $W$.
	
	Take $s \in W$. 
	There exist $u_1,u_2 \in \U$ with $\theta_\U(s) = \omega_{u_2,u_1}$, and $v_1,v_2 \in \V$ with $\theta_\V(\rin{ u_2 }{ \theta_\U(s)(u_1) }{\U}) = \omega_{v_2,v_1}$.
	We obtain $\theta_{\U\otimes\V}(s) = \omega_{u_2 \otimes v_2,u_1 \otimes v_1}$ by simple calculations.
	Thus $\theta_{\U\otimes\V}|_W$ is into $\cpK(\U\otimes\V)$.
	
	For $u_1,u_2\in\U$ and $v_1,v_2\in\V$, there exist $s_2 \in J$ with $\theta_\V(s_2) = \omega_{v_2,v_1}$, and $s_1 \in I$ with $\theta_\U(s_1) = \omega_{u_2s_2,u_1}$.
	For $u,u'\in\U$, we have
	\begin{align*}
	\rin{ u' }{ \theta_\U(s_1)(u) }{\U} &= \rin{ u' }{ \omega_{u_2s_2,u_1}u }{\U}\\
	&= \rin{ u' }{ u_2s_2\rin{ u_1 }{ u }{\U} }{\U}\\
	&= \rin{ u' }{ u_2 }{\U} s_2 \rin{ u_1 }{ u }{\U}.
	\end{align*}
	Since $s_2\in J$ and $J$ is a two-sided ideal, $\rin{ u' }{ \theta_\U(s_1)(u) }{\U} \in J$.
	Thus we have $s_1 \in W$.
	By simple calculations, we obtain $\omega_{u_2\otimes v_2, u_1\otimes v_1} = \theta_{\U\otimes\V}(s_1)$.
	This implies that $\theta_{\U\otimes\V}|_W$ is onto $\cpK(\U\otimes\V)$.
	
	By Lemma \ref{Lemma: left pairing unique}, the inverse correspondence $\U\otimes\V\: S_1 \rightarrow S_3$ has a unique left pairing which makes it a partial Morita equivalence.
	This left pairing satisfies that
	\[
	\lin{ u_2\otimes v_2 }{ u_1\otimes v_1 }{\U\otimes\V} = \lin{ u_2\lin{ v_2 }{ v_1 }{\V} }{ u_1 }{\U}
	\]
	for $u_1,u_2\in\U$ and $v_1,v_2\in\V$.
	This follows from the fact mentioned above such that $\omega_{u_2\otimes v_2, u_1\otimes v_1} = \theta_{\U\otimes\V}(s_1)$ holds.
	
	If $\U$ and $\V$ comes from Morita equivalences, then so does their tensor product $\U\otimes\V$.
	This is nothing but the tensor product of Morita contexts defined in \cite[Proposition 2.5]{Ste11}.
\end{exam}

\begin{exam}\label{Example: tensor semigrp hom}
	Let $S_i$ be inverse semigroups with $i =1,2,3$, and $\theta_i\: S_i \rightarrow S_{i+1}$ be semigroup homomorphisms with $ i =1,2$.
	We can check that the map 
	$\U_{\theta_1} \otimes \U_{\theta_2} \rightarrow \U_{\theta_2\theta_1}; u_1 \otimes u_2 \mapsto \theta_2(u_1)u_2$ 
	becomes bijective, right pairing preserving, and a left $S_1$-map (that is, this is an isomorphism between inverse correspondences mentioned in Subsection \ref{Subsection: A bicategory IScorr of inverse semigroups}).
	Thus the tensor product $\U_{\theta_1} \otimes \U_{\theta_2}$ is isomorphic to the inverse correspondence $\U_{\theta_2\theta_1}$.
\end{exam}

\section{A bicategory $\IScorr$ of inverse semigroups}
\label{Section: A bicategory IScorr of inverse semigroups}

In this section, we will see that inverse semigroups and non-degenerate inverse correspondences form a bicategory.
See \cite{Ben67} or \cite{Lei98} for more details of the bicategory theory.

\subsection{Definition and examples of bicategories}

We fix some conventions and notations in category theory before defining bicategories.
A category consists of objects, morphisms, compositions of morphisms, and identity morphisms.
For objects $x$ and $y$, we denote a morphism $f$ from $x$ to $y$ as $f \: x \rightarrow y$.
For morphisms $f \: x\rightarrow y$ and $g\: y \rightarrow z$, the composition of $f$ and $g$ is denoted as $g\cdot f \: x \rightarrow z$.
We denote the identity morphism for an object $x$ as $1_x\: x \rightarrow x$.
These satisfy the associative law and the unit law. 
A morphism $f\: x \rightarrow y$ is an \emph{isomorphism} if there exists a morphism $g \: y \rightarrow x$ such that $g \cdot f = 1_x$ and $f \cdot g = 1_y$. 
We say that $x$ and $y$ are \emph{isomorphic} if there exists an isomorphism between $x$ and $y$.

\begin{defi}\label{Definition: bicategory}
	A \emph{bicategory} $\biC$ consists of the following date;
	\begin{enumerate}[(i)]
		\item a collection of \emph{$0$-arrows},
		
		\item a category $\biC(x,y)$ for every $0$-arrows $x,y$; 
		an object $f$ of $\biC(x,y)$ is called a \emph{$1$-arrow from $x$ to $y$} and denoted as $f\: x \rightarrow y$; 
		a morphism $\sigma$ from a $1$-arrow $f\:x \rightarrow y$ to $g\:x \rightarrow y$ is called a \emph{$2$-arrow} from $f$ to $g$ and denoted as $\sigma\: f \Rightarrow g\: x \rightarrow y$ or $\sigma \: f \Rightarrow g$;
		the identity morphism for a $1$-arrow $f$ is called a \emph{unit $2$-arrow} and denoted as $1_f\: f \Rightarrow f$.
		
		\item a functor $\circ_{x,y,z}\: \biC(x,y) \times \biC(y,z) \rightarrow \biC(x,z)$ called a \emph{composition functor} for each triplet of $0$-arrows $x,y,z$; 
		the object $\circ_{x,y,z}(f,g)$ is denoted as $g \bullet f \: x \rightarrow z$ for $1$-arrows $f \: x \rightarrow y$, $g \: y \rightarrow z$;
		the morphism $\circ_{x,y,z}(\sigma, \tau)$ is denoted as $\tau \bullet \sigma \: g \bullet f \Rightarrow g' \bullet f' : x \rightarrow z$ for $2$-arrows $\sigma\: f \Rightarrow f'\: x \rightarrow y$, $\tau: g \Rightarrow g'\: y \rightarrow z$,
				
		\item an isomorphic $2$-arrow $\alpha_{f,g,h} \: h \bullet (g \bullet f) \Rightarrow (h \bullet g) \bullet f$ called an \emph{associator} for each triplet of $1$-arrows $f\: x \rightarrow y$, $g\: y \rightarrow z$, $h\: z \rightarrow w$; the associators make the following diagram commute;
		\[
		\begin{tikzcd}
		h \bullet (g \bullet f) & (h \bullet g) \bullet f \\
		h' \bullet (g' \bullet f') & (h' \bullet g') \bullet f',
		\arrow[from = 1-1, to = 1-2, Rightarrow, "\alpha_{f,g,h}"]
		\arrow[from = 2-1, to = 2-2, Rightarrow, "\alpha_{f',g',h'}"]
		\arrow[from = 1-1, to = 2-1, Rightarrow, "\nu \bullet (\tau \bullet \sigma)"']
		\arrow[from = 1-2, to = 2-2, Rightarrow, "(\nu \bullet \tau) \bullet \sigma"]
		\end{tikzcd}
		\]
		
		\item a $1$-arrow $1_x\: x \rightarrow x$ called \emph{unit $1$-arrow for $x$} for each $0$-arrow $x$,
		
		\item an isomorphic $2$-arrow $\lambda_f \: 1_y\bullet f \Rightarrow f$ called a \emph{left unitor}, 
		and an isomorphic $2$-arrow $\rho_f\: f \bullet 1_x \Rightarrow f$ called a \emph{right unitor} for each $1$-arrow $f\: x \rightarrow y$; the left and right unitors make the following diagrams commute;
		\[
		\begin{tikzcd}[ampersand replacement = \&]
		1_y\bullet f \& f\\
		1_y\bullet f' \& f',
		\arrow[from = 1-1, to = 1-2, Rightarrow, "\lambda_f"]
		\arrow[from = 2-1, to = 2-2, Rightarrow, "\lambda_{f'}"]
		\arrow[from = 1-1, to = 2-1, Rightarrow, "1_{1_y}\bullet \sigma"']
		\arrow[from = 1-2, to = 2-2, Rightarrow, "\sigma"]
		\end{tikzcd}
		\begin{tikzcd}[ampersand replacement = \&]
		f\bullet 1_x \& f\\
		f'\bullet 1_x \& f'.
		\arrow[from = 1-1, to = 1-2, Rightarrow, "\rho_f"]
		\arrow[from = 2-1, to = 2-2, Rightarrow, "\rho_{f'}"]
		\arrow[from = 1-1, to = 2-1, Rightarrow, "\sigma\bullet 1_{1_x}"']
		\arrow[from = 1-2, to = 2-2, Rightarrow, "\sigma"]
		\end{tikzcd}
		\]
	\end{enumerate}
	
	These make the following diagrams commute;
	\[
	\begin{tikzcd}[ampersand replacement = \&]
	g \bullet (1_y \bullet f) \&\& (g \bullet 1_y) \bullet f,\\
	\& g \bullet f \&
	\arrow[from = 1-1, to = 1-3, Rightarrow, "\alpha_{f,1_y,g}"]
	\arrow[from = 1-1, to = 2-2, Rightarrow, "1_g\bullet \lambda_f"']
	\arrow[from = 1-3, to = 2-2, Rightarrow, "\rho_g \bullet 1_f"]
	\end{tikzcd}
	\]
	\[
	\begin{tikzcd}[ampersand replacement = \&]
	f_4 \bullet  (f_3 \bullet (f_2 \bullet f_1)) \& (f_4 \bullet  f_3) \bullet (f_2 \bullet f_1) \& ((f_4 \bullet  f_3) \bullet f_2) \bullet f_1\\
	f_4 \bullet  ((f_3 \bullet f_2) \bullet f_1) \&\& (f_4 \bullet  (f_3 \bullet f_2)) \bullet f_1.
	\arrow[from = 1-1, to = 1-2, Rightarrow, "\alpha_{ (f_2 \bullet f_1), f_3, f_4 }"{above = 4pt}]
	\arrow[from = 1-2, to = 1-3, Rightarrow, "\alpha_{ f_1, f_2, (f_4 \bullet f_3) }"{above = 4pt}]
	\arrow[from = 1-1, to = 2-1, Rightarrow, "1_{f_4} \bullet \alpha_{f_1,f_2,f_3}"']
	\arrow[from = 2-3, to = 1-3, Rightarrow, "\alpha_{f_2,f_3,f_4} \bullet 1_{f_1}"']
	\arrow[from = 2-1, to = 2-3, Rightarrow, "\alpha_{f_1, (f_3 \bullet f_2), f_4}"]
	\end{tikzcd}
	\]
\end{defi}

\begin{defi}
	A $1$-arrow $f \: x \rightarrow y$ in a bicategory $\biC$ is said to be an \emph{equivalence} if there exists a $1$-arrow $g \: y \rightarrow x$ such that $g\bullet f$ is isomorphic to $1_x$ and $f \bullet g$ is isomorphic to $1_y$.
	Two $0$-arrows $x,y$ are \emph{equivalent} if there exists an equivalence from $x$ to $y$.
\end{defi} 

\begin{exam}
	The bicategory $\Gr$ of \'etale groupoids and groupoid correspondences is studied in \cite{Alb15} and \cite{AKM22}.
	The composition of groupoid correspondences $\X \: G \rightarrow H$ and $\Y \: H \rightarrow K$ is the groupoid correspondence $\X \circ_H \Y \: G \rightarrow K$ defined in \cite[Subsection 2.3]{Alb15} and \cite[Section 5]{AKM22}.
	
	For two groupoid correspondences $\X,\Y\: G \rightarrow H$, a $2$-arrow between them is a $G,H$-equivariant homeomorphism defined in \cite[p.23]{Alb15} (that is, a homeomorphism which is compatible with left and right actions of $G$ and $H$).
	Two \'etale groupoids are equivalent in this bicategory $\Gr$ if and only if they are Morita equivalent (see \cite[Theorem 2.30]{Alb15}).
	In \cite{AKM22}, the authors allowed injective $G,H$-equivariant continuous maps as $2$-arrows of $\Gr$.
\end{exam}

\begin{exam}
	The bicategory $\Corr$ of $C^*$-algebras and non-degenerate \Ccorrs is studied in \cite{BMZ13}.
	The composition of \Ccorrs $\E \: A \rightarrow B$ and $\F \: B \rightarrow C$ is defined by their interior tensor product $\E \otimes_B \F \: A \rightarrow C$ (see \cite[p.38-44]{Lan95}).
	For two \Ccorrs $\E,\F\: A \rightarrow B$, a $2$-arrow between them is a unitary $A,B$-bimodule map $\sigma \: \E \rightarrow \F$ in \cite{BMZ13}.
	Two $C^*$-algebras are equivalent in this bicategory $\Corr$ if and only if they are Morita equivalent (see \cite[Proposition 2.11]{BMZ13}).
	In \cite{AKM22}, the authors allowed isometric (not necessary invertible or adjointable) $A,B$-bimodule maps as $2$-arrows of $\Corr$.
\end{exam}

\subsection{A bicategory $\IScorr$ of inverse semigroups}
\label{Subsection: A bicategory IScorr of inverse semigroups}

Now we construct a bicategory $\IScorr$ consisting of inverse semigroups and non-degenerate inverse correspondences.
We set inverse semigroups as $0$-arrows.

Let $S$, $T$ be inverse semigroups and $\U$, $\U'$ be inverse correspondences from $S$ to $T$.

\begin{defi}
	A \emph{correspondence map} $\sigma\: \U \rightarrow \U'$ is a right pairing preserving left $S$-map from $\U$ to $\U'$.
\end{defi}

We can easily check that the composition of two correspondence maps is also a correspondence map.
Thus all non-degenerate inverse correspondences from $S$ to $T$ and all correspondence maps form a category with respect to the usual composition $\circ$ of maps and the identity maps $\id_{\U}$ on inverse correspondences $\U$.
We denote this category as $\IScorr(S,T)$.

\begin{lemm}\label{Lemma: bijective corr map}
	A correspondence map $\sigma\: \U \rightarrow \U'$ is an isomorphism in the category $\IScorr(S,T)$ if and only if $\sigma$ is surjective.
\end{lemm}
\begin{proof}
	We can easily check that a correspondence map is an isomorphism in this category $\IScorr(S,T)$ if and only if it is bijective.
	By Lemma \ref{Lemma: pairing preserving injective}, every correspondence map is injective.
\end{proof}

\begin{lemm}
	For an isomorphism $\iota\:\U\rightarrow\U'$ between inverse correspondences, if $\U$ and $\U'$ are partial Morita equivalences, then $\iota$ preserves the left pairings.
\end{lemm}
\begin{proof}
	This follows from Lemma \ref{Lemma: left pairing unique}.
\end{proof}

Let $S_i$ be an inverse semigroup with $i = 1,2,3$, and $\U_i, \U'_i \: S_i \rightarrow S_{i+1}$ be inverse correspondences with $i = 1,2$.
For correspondence maps $\sigma_i\: \U_i \rightarrow \U'_i$ for $i =1,2$, the \emph{tensor product} $\sigma_1 \otimes \sigma_2 \: \U_1 \otimes \U_2 \rightarrow \U'_1 \otimes \U'_2$ of correspondence maps is defined as 
\[
(\sigma_1 \otimes \sigma_2)(u_1 \otimes u_2) := \sigma_1(u_1) \otimes \sigma_2(u_2)
\]
for every $u_i\in \U_i$ with $i = 1,2$.

\begin{lemm}
	For correspondence maps $\sigma_i\: \U_i \rightarrow \U'_i$ for $i =1,2$, $\sigma_1\otimes\sigma_2$ is well-defined and a correspondence map.
\end{lemm}
\begin{proof}
	For every $u_1\in \U_1$, $u_2 \in \U_2$, and $s_2 \in S_2$, 
	\[
	(\sigma_1 \otimes \sigma_2) ( u_1 \otimes s_2u_2 ) = (\sigma_1 \otimes \sigma_2)(u_1s_2 \otimes u_2)
	\]
	holds since $\sigma_1$ is a right $S_2$-map and $\sigma_2$ is a left $S_2$-map.
	Thus $\sigma_1 \otimes \sigma_2$ is well-defined.
	The map $\sigma_1 \otimes \sigma_2$ is a left $S_1$-map because $\sigma_1$ is a left $S_1$-map.
	This map is right pairing preserving because $\sigma_1$ and $\sigma_2$ are right pairing preserving and $\sigma_2$ is a left $S_2$-map.
\end{proof}

\begin{lemm}
	For inverse semigroups $S_1,S_2$, and $S_3$, the tensor product of inverse correspondences and the tensor product of correspondence maps form a functor 
	\[
	\otimes_{S_1,S_2,S_3} \: \IScorr(S_1,S_2) \times \IScorr(S_2,S_3) \rightarrow \IScorr(S_1,S_3).
	\]
\end{lemm}
\begin{proof}
	If two inverse correspondences are non-degenerate, then so is their tensor product by Proposition \ref{Proposition: tensor inverse corr}.
	We can easily check that 
	\[
	(\sigma'_1 \otimes \sigma'_2) \circ (\sigma_1 \otimes \sigma_2) = (\sigma'_1 \circ \sigma_1) \otimes (\sigma'_2 \circ \sigma_2)
	\]
	for inverse correspondences $\U_i, \U'_i, \U''_i \: S_i \rightarrow S_{i+1}$, and correspondence maps $\sigma_i\: \U_i \rightarrow \U'_i$, $\sigma'_i\: \U'_i \rightarrow \U''_i$ for $i = 1,2$.
	The tensor product of the identity maps $1_{\U_1}$ and $1_{\U_2}$ is the identity map for $\U_1 \otimes \U_2$.
\end{proof}

\begin{lemm}
	For inverse correspondences $\U_i\: S_i \rightarrow S_{i+1}$ with $i=1,2,3$, a map defined as 
	\[
	\alpha_{\U_1,\U_2,\U_3} \: \U_1 \otimes (\U_2 \otimes \U_3) \rightarrow (\U_1 \otimes \U_2) \otimes \U_3 ; u_1\otimes (u_2 \otimes u_3) \mapsto (u_1\otimes u_2) \otimes u_3
	\]
	is well-defined and an isomorphic correspondence map.
	This map is natural for $\U_1$, $\U_2$, and $\U_3$, that is, for every correspondence maps $\sigma_i \: \U_i \rightarrow \U_i'$ for $i=1,2,3$, the following diagram commutes;
	\[
	\begin{tikzcd}
	\U_1 \otimes (\U_2 \otimes \U_3) & (\U_1 \otimes \U_2) \otimes \U_3\\
	\U'_1 \otimes (\U'_2 \otimes \U'_3) & (\U'_1 \otimes \U'_2) \otimes \U'_3.
	\arrow[from = 1-1, to = 1-2, rightarrow, "\alpha_{\U_1,\U_2,\U_3}"]
	\arrow[from = 2-1, to = 2-2, rightarrow, "\alpha_{\U'_1,\U'_2,\U'_3}"]
	\arrow[from = 1-1, to = 2-1, rightarrow, "\sigma_1 \otimes (\sigma_2 \otimes \sigma_3)"']
	\arrow[from = 1-2, to = 2-2, rightarrow, "(\sigma_1 \otimes \sigma_2) \otimes \sigma_3"]
	\end{tikzcd}
	\]
\end{lemm}
\begin{proof}
	Straightforward.
\end{proof}

As seen in Example \ref{Example: tensor semigrp hom}, an inverse semigroup $S$ can be regarded as a non-degenerate inverse correspondence from $S$ to $S$.
We set this correspondence $S$ as a unit $1$-arrow for $S$.

\begin{lemm}
	For an inverse correspondence $\U$ from an inverse semigroup $S$ to $T$, the map
	\[
	\rho_{\U} \: \U \otimes T \rightarrow \U; u \otimes t \mapsto ut
	\]
	is an isomorphic correspondence map.
	If $\U$ is non-degenerate, then the map 
	\[
	\lambda_{\U} \: S\otimes \U \rightarrow \U; s \otimes u \mapsto su
	\]
	is an isomorphic correspondence map.
	These maps make the following diagrams commute;
	\[
	\begin{tikzcd}
	\U\otimes T & \U \\
	\U'\otimes T & \U',
	\arrow[r,from = 1-1, to = 1-2, "\rho_{\U}"]
	\arrow[r,from = 2-1, to = 2-2, "\rho_{\U'}"]
	\arrow[r,from = 1-1, to = 2-1, "\sigma\otimes\id_T"']
	\arrow[r,from = 1-2, to = 2-2, "\sigma"]
	\end{tikzcd}
	\begin{tikzcd}
	S\otimes\U & \U \\
	S\otimes\U'& \U',
	\arrow[r,from = 1-1, to = 1-2, "\lambda_{\U}"]
	\arrow[r,from = 2-1, to = 2-2, "\lambda_{\U'}"]
	\arrow[r,from = 1-1, to = 2-1, "\id_S\otimes\sigma"']
	\arrow[r,from = 1-2, to = 2-2, "\sigma"]
	\end{tikzcd}
	\]
	for every inverse correspondences $\U,\U'\: S \rightarrow T$, and every correspondence map $\sigma\:\U\rightarrow\U'$.
\end{lemm}
\begin{proof}
	By Lemma \ref{Lemma: bijective corr map}, it is enough to check that $\lambda_{\U}$ and $\rho_{\U}$ are surjective correspondence maps.
	
	The map $\rho_{\U} \: \U \otimes T \rightarrow \U$ is right pairing preserving by (R-i) in \ref{Definition: inverse set} and the right version of Lemma \ref{Lemma: left pairing basics} (ii), and is a left $S$-map by Lemma \ref{Lemma: left action basics} (ii).
	The surjectivity of $\rho_{\U}$ follows from (R-iii) in Definition \ref{Definition: inverse set}.
	 
	 The map $\lambda_{\U}$ is right pairing preserving since $\theta_{\U}(s^*) = \theta_{\U}(s)\d$ holds for every $s \in S$, and is a left $S$-map clearly.
	 The surjectivity of $\lambda_{\U}$ follows from non-degeneracy of $\U$.
	 
	 The diagrams commute because $\sigma$ is a correspondence map.
\end{proof}

\begin{theo}\label{Theorem: IScorr}
	The above date form a bicategory $\IScorr$ of inverse semigroups and non-degenerate inverse correspondences.
\end{theo}
\begin{proof}
	We need to see that the triangle diagram and the pentagon diagram in Definition \ref{Definition: bicategory} commute.
	It is easy to check these.
\end{proof}

We investigate equivalences in the bicategory $\IScorr$.
Let $S,T$ be inverse semigroups.

For a right regular $T$-set $\U$, we can obtain a left regular $T$-set $\widetilde{\U}$ as follows:
We define a set $\widetilde{\U}$ as the set of all symbols $\widetilde{u}$ running over all $u \in \U$.
The left action of $T$ and the left pairing is defined as 
\[
t\widetilde{u} := \widetilde{ut^*}, \text{ and } \lin{ \widetilde{u_1} }{ \widetilde{u_2} }{\widetilde{\U}} := \rin{ u_1 }{ u_2 }{\U}
\]
for every $t \in T$ and $u,u_1,u_2\in\U$.
The set $\widetilde{\U}$ becomes a left regular $T$-set with respect to the above structures.
If $\U$ is a right inverse $T$-set, then $\widetilde{\U}$ is a left inverse $T$-set.

For a partial Morita equivalence $\U$ from $S$ to $T$, we obtain a right action of $S$ and a right pairing on the left inverse $T$-set $\widetilde{\U}$ as follows:
For every $u,u_1,u_2\in\U$ and $s \in S$,
\[
\widetilde{u}s := \widetilde{s^*u}, \text{ and } \rin{ \widetilde{u_1} }{ \widetilde{u_2} }{\widetilde{\U}} := \lin{ u_1 }{ u_2 }{\U}.
\]
The left inverse $T$-set $\widetilde{\U}$ becomes a partial Morita equivalence from $T$ to $S$ with respect to these structures.
If $\U$ is a Morita equivalence, then so is $\widetilde{\U}$.

\begin{rema}
	An inverse correspondence $\U\: S \rightarrow T$ becomes a generalized heap with respect to a ternary operation $\{,,\} \: \U\times\U\times\U\rightarrow\U$ defined by $\{u_1,u_2,u_3\}_{\U} := u_1\rin{ u_2 }{ u_3 }{\U}$ for $u_1,u_2,u_3 \in \U$ (see \cite[p.318]{Law11} for the definition of generalized heaps). 
	This ternary operation satisfy that $\{ u_1, su_2, u_3 \}_{\U} = \{ u_1, u_2, s^*u_3 \}_{\U}$ and $\{ u_1 , u_2t , u_3 \}_{\U} = \{ u_1t^* , u_2 , u_3 \}_{\U}$ for $u_1,u_2,u_3 \in \U$, $s \in S$ and $t \in T$.
	These equations are analogue to the axioms of generalized correspondences in the $C^*$-algebra theory (see \cite[p.5]{Exe07}).
	The left inverse $T$-set $\widetilde{\U}$ also becomes a generalized heap with respect to a ternary operation $\{,,\}_{\widetilde{\U}} \: \widetilde{\U} \times \widetilde{\U} \times \widetilde{\U} \rightarrow \widetilde{\U}$ defined by $\{ \widetilde{u_1}, \widetilde{u_2}, \widetilde{u_3} \}_{\widetilde{\U}} := \lin{ \widetilde{u_1} }{ \widetilde{u_2} }{\widetilde{\U}}\widetilde{u_3} $.
	We can define a right action of $S$ on the generalized heap $\widetilde{\U}$ in the same way as the case of partial Morita equivalence mentioned above.
	The ternary operation satisfy that $\{ \widetilde{u_1}, t\widetilde{u_2}, \widetilde{u_3}\} = \{ \widetilde{u_1}, \widetilde{u_2}, t^*\widetilde{u_3}\}$ and $\{ \widetilde{u_1}, \widetilde{u_2}s, \widetilde{u_3}\} = \{ \widetilde{u_1}s^*, \widetilde{u_2}, \widetilde{u_3}\}$ for $\widetilde{u_1},\widetilde{u_2},\widetilde{u_3}\in\widetilde{\U}$, $t \in T$, and $s \in S$.
\end{rema}

Let $\U$ be a partial Morita equivalence from $S$ to $T$.
We recall that $\lin{ \U }{ \U }{\U}$ is a two-sided ideal of $S$.
Hence this becomes a partial Morita equivalence from $S$ to $S$.
The two-sided ideal $\rin{ \U }{\U }{\U}$ of $T$ becomes a partial Morita equivalence from $T$ to $T$.
\begin{prop}\label{Proposition: U and Utilde}
	For a partial Morita equivalence $\U$ from $S$ to $T$, we have 
	\[
	\U \otimes \widetilde{\U} \simeq \lin{ \U }{ \U }{\U} \text{ and }\ \widetilde{\U} \otimes \U \simeq \rin{ \U }{ \U }{\U},
	\]
	as partial Morita equivalences.
\end{prop}
\begin{proof}
	We define a map $\iota\: \widetilde{\U} \otimes \U \rightarrow \rin{ \U }{ \U }{\U}$ as 
	\[
	\iota( \widetilde{u} \otimes u' ) := \rin{ u }{ u' }{\U}
	\]
	for $u,u'\in\U$.
	It is clear that $\iota$ is surjective.
	For $t \in T$ and $u,u'\in\U$, we have
	\begin{align*}
	\iota( t(\widetilde{u} \otimes u') ) 
	&= \iota( (t\widetilde{u}) \otimes u' )\\
	&= \iota\big( \widetilde{ut^*} \otimes u' \big)\\
	&= \rin{ ut^* }{ u' }{\U}\\
	&= t\rin{ u }{ u' }{\U}\\
	&= t\iota( \widetilde{u} \otimes u' ).
	\end{align*}
	Thus $\iota$ is a left-$T$ map.
	For $u_1,u_1',u_2,u_2'\in\U$, we have
	\begin{align*}
	\rin{ \widetilde{u_1} \otimes u_1' }{ \widetilde{u_2} \otimes u_2' }{\widetilde{\U}\otimes\U} 
	&= \rin{ u_1' }{ \rin{ \widetilde{u_1} }{ \widetilde{u_2} }{\widetilde{\U}}u_2' }{\U}\\
	&= \rin{ u_1' }{ \lin{ u_1 }{ u_2 }{\U}u_2' }{\U}\\
	&= \rin{ u_1' }{ u_1\rin{ u_2 }{ u_2' }{\U} }{\U}\\
	&= \rin{ u_1' }{ u_1 }{\U}\rin{ u_2 }{ u_2' }{\U}\\
	&= \rin{ u_1 }{ u_1' }{\U}^*\rin{ u_2 }{ u_2' }{\U}\\
	&= \rin{ \iota(\widetilde{u_1} \otimes u_1') }{ \iota(\widetilde{u_2} \otimes u_2') }{\rin{\U}{\U}{\U}}.
	\end{align*}
	Thus $\iota$ is right pairing preserving.
	By Lemma \ref{Lemma: bijective corr map}, $\iota$ is an isomorphism.
	
	In a similar way, we obtain that the map $\U\otimes\widetilde{\U} \rightarrow \lin{ \U }{ \U }{\U}; u' \otimes \widetilde{u} \mapsto \lin{ u' }{ u }{\U}$ is an isomorphism. 
\end{proof}

The following proposition is analogues to \cite[Lemma 2.4]{EKQR06}.
\begin{prop}\label{Proposition: equiv is Morita}
	Let $\U\: S \rightarrow T$ and  $\V\: T\rightarrow S$ be non-degenerate inverse correspondences.
	If $\U\otimes \V$ is isomorphic to $S$ and $\V \otimes \U$ is isomorphic to $T$ as inverse correspondences, then $\U$ and $\V$ are Morita equivalence.
\end{prop}
\begin{proof}
	By Lemma \ref{Lemma: full}, $\U$ and $\V$ are right full.	
	By symmetry, we only see that $\U$ becomes a Morita equivalence.
	To prove this, it is enough to show that $\theta_\U\: S \rightarrow \adL(\U)$ is an isomorphism onto $\cpK(\U)$ by Corollary \ref{Corollary: partial Morita iff...} (ii).
	
	For $s_1,s_2\in S$ with $\theta_{\U}(s_1) = \theta_{\U}(s_2)$, we get $s_1(u\otimes v) = s_2(u\otimes v)$ for all $u \in \U$ and $v \in \V$.
	This implies that $s_1s = s_2s$ for all $s \in S$ because $\U\otimes\V$ is isomorphic to $S$ as inverse correspondences.
	By Lemma \ref{Lemma: right cancel}, we have $s_1 = s_2$.
	Thus $\theta_{\U} \: S \rightarrow \adL(\U)$ is injective.
	
	We put the isomorphism from $\V \otimes \U$ to $T$ as $\iota$.
	To show that $\theta_{\U}(S) = \cpK(\U)$, we construct a surjection from $\V$ to $\U$.
	We define a map $\Phi\: \V \rightarrow \U$ as the unique map which satisfies that 
	\[
	\rin{ \Phi(v) }{ u }{\U} = \iota(v\otimes u)
	\]
	for every $u \in \U$.
	The uniqueness of $\Phi$ is clear by Lemma \ref{Lemma: non-degenerate}.
	We can construct $\Phi$ as follows:
	Every element of $\V$ is in the form of $\iota(v'\otimes u')^*v$ with some $u' \in \U$ and $v,v'\in V$ because $\V$ is non-degenerate and $\iota$ is surjective.
	We define $\Phi(\iota(v'\otimes u')^*v) := \rin{v}{v'}{} u'$ for every $v,v'\in\V$ and $u' \in \U$.
	We get 
	\begin{align*}
	\rin{ \rin{ v }{ v' }{\V}u' }{ u }{\U} 
	&= \rin{ u' }{ \rin{ v' }{ v }{\V} u }{\U}\\
	&= \rin{ v' \otimes u' }{ v\otimes u }{\V\otimes\U}\\
	&= \iota(v' \otimes u')^*\iota(v \otimes u)\\
	&= \iota( \iota(v' \otimes u')^* v \otimes u )
	\end{align*}
	for every $u \in \U$.
	This implies that $\Phi$ is well-defined and satisfies $\rin{ \Phi(v) }{ u }{\U} = \iota(v\otimes u)$ for every $u \in \U$ and $v \in \V$.
	Every element of $\U$ is in the form of $\rin{ v }{ v' }{\V}u'$ with some $v,v' \in \V$ and $u'\in\U$ because $\U$ is non-degenerate and $\V$ is right full.
	Thus we get that $\Phi$ is surjective.
	
	Take an element $k$ of $\cpK(\U)$.
	There exist $v_1,v_2\in\V$ with $k = \omega_{\Phi(v_1),\Phi(v_2)}$ because $\Phi$ is surjective.
	For every $u_1,u_2\in\U$, we get 
	\begin{align*}
		\rin{ u_1 }{ \omega_{\Phi(v_1),\Phi(v_2)} u_2 }{\U} 
		&= \rin{ u_1 }{ \Phi(v_1)\rin{\Phi(v_2)}{u_2}{\U} }{\U} \\
		&= \rin{ u_1 }{ \Phi(v_1) }{\U} \rin{\Phi(v_2)}{u_2}{\U} \\
		&= \iota( v_1 \otimes u_1)^* \iota( v_2 \otimes u_2 ) \\
		&= \rin{ v_1 \otimes u_1 }{ v_2 \otimes u_2 }{\V\otimes\U}\\
		&= \rin{ u_1 }{ \rin{ v_1 }{ v_2 }{\V} u_2 }{\U}.
	\end{align*}
	Thus $k = \theta_{\U}( \rin{ v_1 }{v_2}{\V} ) \in \theta_{\U}(S)$ holds.
	
	Take an element $s \in S$.
	There exists $v_1,v_2 \in \V$ with $s = \rin{v_1}{v_2}{\V}$ because $\V$ is right full.
	By the above discussion, we get $\theta_\U(s) = \omega_{\Phi(v_1),\Phi(v_2)} \in \cpK(\U)$.
	Thus we get $\theta_\U(S) = \cpK(\U)$.
\end{proof}

\begin{theo}\label{Theorem: Morita iff equiv}
	For a non-degenerate inverse correspondence $\U\:S\rightarrow T$, $\U$ is a Morita equivalence if and only if $\U$ is an equivalence in the bicategory $\IScorr$.
\end{theo}
\begin{proof}
	For a Morita equivalence $\U\: S \rightarrow T$, the Morita equivalence $\widetilde{\U}\: T \rightarrow S$ satisfies $\U\otimes\widetilde{\U} \simeq S$ and $\widetilde{\U}\otimes\U\simeq T$ by Proposition \ref{Proposition: U and Utilde}.
	Thus $\U$ is an equivalence in the bicategory $\IScorr$.
	The if part follows from Proposition \ref{Proposition: equiv is Morita}.
\end{proof}

\section{Multiplier semigroups}
\label{Section: Multiplier semigroups}
In the $C^*$-algebra theory, the \emph{multiplier algebras} $M(A)$ of $C^*$-algebras $A$ are studied well; see \cite{Bus68}, \cite{APT73} or \cite{Lan95} for example.
In this section, we define the multiplier semigroups $M(S)$ of inverse semigroups $S$ as an analogy of the multiplier algebras, and show that every inverse semigroup has its multiplier semigroup.

Let $S$ be an inverse semigroup.

\begin{defi}\label{Definition: multiplier}
	The \emph{multiplier semigroup} $M(S)$ of $S$ is an inverse semigroup which includes $S$ as a two-sided ideal and satisfies the following universality: 
	For every inverse semigroup $\widetilde{S}$ which includes $S$ as a two-sided ideal, there exists a unique semigroup homomorphism $\theta \: \widetilde{S} \rightarrow M(S)$ such that the following diagram commutes;
	\[
	\begin{tikzcd}
	\widetilde{S} &\\
	S & M(S).
	\arrow[from = 2-1, to = 1-1, "\hookrightarrow", sloped, phantom]
	\arrow[from = 2-1, to = 2-2, "\hookrightarrow"', phantom]
	\arrow[from = 1-1, to = 2-2, "\theta", dashed]
	\end{tikzcd}
	\]
\end{defi}
We can check easily that the multiplier semigroup is unique up to isomorphism if it exists.
Before proving that there exist the multiplier semigroups for all inverse semigroups, we investigate the universality of the inverse semigroup of adjointable maps.

Let $S, T$ be inverse semigroups, and $\U$ be an inverse $T$-set.
We say that a semigroup homomorphism $\theta \: S \rightarrow \adL(\U)$ is \emph{non-degenerate} if every element $u \in \U$ is in the form of $\theta(s')u'$ with some $s' \in S$ and $u'\in \U$ (that is, $\U$ and $\theta$ form a non-degenerate inverse correspondence from $S$ to $T$).

\begin{prop}\label{Proposition: the universality of L(U)}
	Let $\theta\: S \rightarrow \adL(\U)$ be a non-degenerate semigroup homomorphism.
	For every inverse semigroup $\widetilde{S}$ which includes $S$ as a two-sided ideal, there exists a unique semigroup homomorphism $\widetilde{\theta} \: \widetilde{S} \rightarrow \adL(\U)$ such that the following diagram commutes;
	\[
	\begin{tikzcd}
	\widetilde{S} &\\
	S & \adL(\U).
	\arrow[from = 2-1, to = 1-1, "\hookrightarrow", sloped, phantom]
	\arrow[from = 2-1, to = 2-2, "\theta"']
	\arrow[from = 1-1, to = 2-2, "\widetilde{\theta}", dashed]
	\end{tikzcd}
	\]
\end{prop}
\begin{proof}
	For every $s_0 \in \widetilde{S}$, we define a map $\widetilde{\theta}(s_0)\: \U \rightarrow \U$ as 
	\[
	\widetilde{\theta}(s_0)(\theta(s)u) := \theta(s_0s)(u)
	\]
	for $s \in S$ and $u \in \U$.
	We can see that this map $\widetilde{\theta}(s_0)$ is well-defined and adjointable by the following calculation:
	For $s,s' \in S$ and $u,u' \in \U$, we have
	\begin{align*}
	\rin{ \theta(s)u }{ \theta(s_0s')(u') }{\U} 
	&= \rin{ u }{ \theta(s)^*\theta(s_0s')(u') }{\U}\\
	&= \rin{ u }{ \theta(s^*s_0s')(u') }{\U}\\
	&= \rin{ u }{ \theta(s_0^*s)^*\theta(s')(u') }{\U}\\
	&= \rin{ \theta(s_0^*s)(u) }{ \theta(s')(u') }{\U}.
	\end{align*}
	It is easy to see that the map $\widetilde{\theta} \: \widetilde{S} \rightarrow \adL(\U)$ is a semigroup homomorphism, and that the restriction of $\widetilde{\theta}$ to $S$ coincides with $\theta$.
\end{proof}

We recall that $S$ produces the inverse $S$-set $S$ and the inverse semigroup $\adL(S)$.
\begin{lemm}\label{Lemma: lambda}
	\begin{enumerate}[(i)]
		\item For $s \in S$, the map $\lambda_s\: S \rightarrow S; s' \mapsto ss'$ is adjointable, and $\lambda_s\d = \lambda_{s^*}$ holds.
		\item The map $\lambda\: S \rightarrow \adL(S) ; s \mapsto \lambda_s$ is an injective semigroup homomorphism.
		\item For $s \in S$ and $\varphi\in\adL(S)$, we have $\varphi\lambda_s = \lambda_{\varphi(s)}$ and $\lambda_s\varphi = \lambda_{\varphi\d(s^*)^*}$.
		\item For $\varphi, \varphi'\in\adL(S)$, if $\varphi\lambda_s = \varphi\lambda_s'$ for all $s\in S$, then $\varphi = \varphi'$ holds.
	\end{enumerate}
\end{lemm}
\begin{proof}
	It is easy to show (i) and that the map $\lambda\: S \rightarrow \adL(S)$ is a semigroup homomorphism.
	The map $\lambda$ is injective by Lemma \ref{Lemma: right cancel}.
	For $s,s' \in S$ and $\varphi\in\adL(S)$, we have 
	\begin{align*}
	\varphi\lambda_s(s') &= \varphi(ss') = \varphi(s)s' = \lambda_{\varphi(s)}(s')\\
	\lambda_s\varphi(s') &= s\varphi(s') = \rin{ s^* }{ \varphi(s') }{T} \\
	&= \rin{ \varphi\d(s^*) }{ s' }{T} = \varphi\d(s^*)^*s' = \lambda_{\varphi\d(s^*)^*}(s'),
	\end{align*}
	where the second equal follows from Lemma \ref{Lemma: adj T-map}.
	Thus $\varphi\lambda_s = \lambda_{\varphi(s)}$ and $\lambda_s\varphi = \lambda_{\varphi\d(s^*)^*}$ hold.
	For $\varphi, \varphi'\in\adL(S)$ such that $\varphi\lambda_s = \varphi\lambda_s'$ for all $s\in T$, we have 
	$\lambda_{\varphi(s)} = \varphi\lambda_s = \varphi'\lambda_s = \lambda_{\varphi'(s)}$.
	The injectivity of $\lambda$ implies $\varphi(s) = \varphi'(s)$. 
	Thus we have $\varphi = \varphi'$.
\end{proof}

\begin{prop}
	For every inverse semigroup $S$, there exists the multiplier semigroup $M(S)$ of $S$.
\end{prop}
\begin{proof}
	The inverse semigroup $\adL(S)$ includes $S$ as a two-sided ideal through the semigroup homomorphism $\lambda \: S \rightarrow \adL(S)$ by Lemma \ref{Lemma: lambda}.
	By Proposition \ref{Proposition: the universality of L(U)}, $\adL(S)$ has the universality in Definition \ref{Definition: multiplier}.
\end{proof}

\begin{rema}
	An \emph{identity element} of an inverse semigroup $S$ is an element $1\in S$ such that $1s = s = s1$ for all $s \in S$.
	This is unique if it exists.
	For every inverse semigroup $S$, $M(S)$ has the identity element because $\adL(S)$ has the identity map on $S$ as the identity element. 
	We can easily check that $S = M(S)$ if and only if $S$ has the identity element.
\end{rema}

\begin{rema}
	For a semigroup $S$, the semigroup of double centralizers on $S$ are introduced in \cite{Joh64}:
	The \emph{double centralizer} on $S$ is a couple of maps $\lambda\:S\rightarrow S$ and $\rho\:S\rightarrow S$ with $s_1\lambda(s_2) = \rho(s_1)s_2$ for $s_1,s_2\in S$.
	The set of all double centralizers $D(S)$ on $S$ becomes a semigroup with respect to the multiplication $(\lambda_1,\rho_1)(\lambda_2,\rho_2) = (\lambda_1\lambda_2,\rho_2\rho_1)$ for $(\lambda_1,\rho_1),(\lambda_2,\rho_2)\in D(S)$.
	We can easily check that a couple of maps $\lambda_s\: S \rightarrow S; s' \mapsto ss'$ and $\rho_s\: S \rightarrow S; s' \mapsto s's$ becomes a double centralizer for $s \in S$.
	The map $S \rightarrow D(S); s \mapsto (\lambda_s,\rho_s)$ is a semigroup homomorphism.
	If $S$ is inverse, this map becomes injective by Lemma \ref{Lemma: right cancel}.
	We can see that for an inverse semigroup $S$ the semigroup $D(S)$ is isomorphic to the inverse semigroup $\adL(S)$ as follows:
	For every $(\lambda,\rho) \in M(S)$, a map $S \ni s \mapsto \rho(s^*)^* \in S$ between the inverse $S$-set $S$ is the adjoint of $\lambda$ because
	\[
	\rin{ s_1 }{ \lambda(s_2) }{S} = s_1^*\lambda(s_2) = \rho(s_1^*)s_2 = \rin{ \rho(s_1^*)^* }{ s_2 }{S}
	\]
	for $s_1,s_2\in S$.
	Thus a map $\iota\:D(S)\ni (\lambda,\rho) \mapsto \lambda \in \adL(S)$ is well-defined.
	This map is a semigroup homomorphism obviously.
	We can see that for $\varphi \in \adL(S)$, a couple of $\varphi$ and a map $\varphi'\: S \rightarrow S; s \mapsto \varphi\d(s^*)^*$ is a double centralizer on $S$, and that a map $\adL(S) \rightarrow D(S); \varphi \mapsto (\varphi,\varphi')$ becomes the inverse of $\iota$.
	Thus we have $D(S) \simeq \adL(S)$.
	Especially, for every inverse semigroup $S$, double centralizers $D(S)$ becomes an inverse semigroup.
	This is another description of the multiplier semigroup of $S$.
	The fact that $D(S)$ becomes an inverse semigroup is analogue to the fact that the set of all double centralizers on a $C^*$-algebra $A$ becomes a $C^*$-algebra proved in \cite{Bus68}.
\end{rema}

We give the condition such that $\widetilde{\theta}$ in Proposition \ref{Proposition: the universality of L(U)} becomes injective.

\begin{defi}
	A two-sided ideal $I$ of $S$ is \emph{essential} if for every $s,s' \in S$, $st = s't$ for all $t \in I$ implies $s = s'$.
\end{defi}

\begin{lemm}
	For a two-sided ideal $I$ of $S$, the following are equivalent:
	\begin{enumerate}[(i)]
		\item $I$ is essential.
		\item For every $s,s' \in S$, $ts = ts'$ for all $t \in I$ implies $s = s'$.
	\end{enumerate}
\end{lemm}
\begin{proof}
	Let $I$ be an essential two-sided ideal of $S$.
	Since $I$ is a two-sided ideal, $I$ becomes an inverse subsemigroup of $S$.
	Take $t_0 \in T$ and $s,s' \in S$ such that $ts = ts'$ for all $t$ arbitrarily.
	We have $tst_0 = ts't_0$ for all $t \in T$.
	By Lemma \ref{Lemma: right cancel}, we get $st_0 = s't_0$.
	Since $I$ is essential, $s = s'$ holds.
	We can prove the converse implication similarly.
\end{proof}

\begin{lemm}\label{Lemma: ess inj}
	Let $T$ be an inverse semigroup, $\widetilde{S}$ be an inverse semigroup which includes $S$ as an essential two-sided ideal, $\theta\: S \rightarrow T$ be a semigroup homomorphism, and $\widetilde{\theta} \: \widetilde{S} \rightarrow T$ be a semigroup homomorphism whose restriction to $S$ coincides with $\theta$;
	\[
	\begin{tikzcd}
	\widetilde{S} &\\
	S & T.
	\arrow[from = 2-1, to = 1-1, "\hookrightarrow", sloped, phantom]
	\arrow[from = 2-1, to = 2-2, "\theta"']
	\arrow[from = 1-1, to = 2-2, "\widetilde{\theta}"]
	\end{tikzcd}
	\]
	The semigroup homomorphism $\widetilde{\theta}$ is injective if and only if so is $\theta$.
\end{lemm}
\begin{proof}
	The only if part is clear.
	Take $s_0,s_0'\in \widetilde{S}$ with $\widetilde{\theta}(s_0) = \widetilde{\theta}(s_0')$.
	For all $s \in S$, we have $\theta(s_0s) = \widetilde{\theta}(s_0)\theta(s) = \widetilde{\theta}(s_0')\theta(s) = \theta(s_0's)$.
	By the injectivity of $\theta$, $s_0s = s_0's$ holds.
	We obtain $s_0 = s_0'$ because $S$ is an essential two-sided ideal of $\widetilde{S}$.
	Thus $\widetilde{\theta}$ is injective.
\end{proof}

\begin{lemm}\label{Lemma: M(S) ess}
	An inverse semigroup $S$ is an essential two-sided ideal of $M(S)$.
\end{lemm}
\begin{proof}
	By Lemma \ref{Lemma: lambda} (iv), the inverse semigroup $\adL(S)$ which is isomorphic to the multiplier semigroup $M(S)$ of $S$ includes $S$ as an essential two-sided ideal. 
\end{proof}

\begin{coro}\label{Corollary: M(S)}
	The multiplier semigroup $M(S)$ of $S$ is the largest inverse semigroup in which $S$ is an essential two-sided ideal, where ``largest'' means that for every inverse semigroup $\widetilde{S}$ which includes $S$ as an essential two-sided ideal, there exists a unique injective semigroup homomorphism $\theta \: \widetilde{S} \rightarrow M(S)$ such that the diagram in Definition \ref{Definition: multiplier} commutes.
\end{coro}
\begin{proof}
	This follows from Lemma \ref{Lemma: M(S) ess} and \ref{Lemma: ess inj}.
\end{proof}

We now characterize the multiplier semigroup in terms of the notion of idealizers in the inverse semigroup of adjointable maps.
Let $U$ be a subsemigroup of $S$.
We define the \emph{idealizer} $I(U)$ of $U$ in $S$ as the largest subsemigroup of $S$ in which $U$ is a two-sided ideal.
It is easy to see the following;
\[
I(U) = \{ s\in S \mid sU \subset U, Us \subset U \}.
\]
\begin{prop}\label{Proposition: idealizer}
	Let $\theta\: S \rightarrow \adL(\U)$ be an injective non-degenerate semigroup homomorphism.
	The semigroup homomorphism in Proposition \ref{Proposition: the universality of L(U)} becomes an isomorphism from the multiplier semigroup $M(S)$ to the idealizer $I(\theta(S))$ of $\theta(S)$ in $\adL(\U)$;
	\[
	\begin{tikzcd}
	M(S) &\\
	S & \adL(\U).
	\arrow[from = 2-1, to = 1-1, "\hookrightarrow", sloped, phantom]
	\arrow[from = 2-1, to = 2-2, "\theta"']
	\arrow[from = 1-1, to = 2-2, "\widetilde{\theta}", dashed]
	\end{tikzcd}
	\]
\end{prop}
\begin{proof}
	We first show that $I(\theta(S))$ includes $\theta(S)$ as an essential two-sided ideal.
	Take $\varphi,\varphi' \in I(\theta(S))$ such that $\varphi\theta(s) = \varphi'\theta(s)$ for all $s \in S$.
	This implies $\varphi(\theta(s)u) = \varphi'(\theta(s)u)$ for all $u \in \U$.
	Since $\theta$ is non-degenerate, $\varphi = \varphi'$ holds.
	Thus $\theta(S)$ is an essential two-sided ideal in $I(\theta(S))$.
	The idealizer $I(\theta(S))$ includes $S$ as an essential ideal through the injective semigroup homomorphism $\theta$.
 	
	By Corollary \ref{Corollary: M(S)}, we obtain the injective semigroup homomorphism $\iota$ from $I(\theta(S))$ to $M(S)$ whose restriction to $S$ is the inclusion map from $S$ to $M(S)$.
	
	The restriction of the composition $\iota\widetilde{\theta}$ to $S$ is the identity map on $S$.
	By the universality of $M(S)$, we see that $\iota\widetilde{\theta}$ is the identity map on $M(S)$.
	 
	 Since $\iota$ is injective and $\iota\widetilde{\theta}$ is the identity map on $M(S)$, we obtain that $\widetilde{\theta}$ is an isomorphism from $M(S)$ to $I(\theta(S))$.
\end{proof}

As an example of multiplier semigroup, we calculate the multiplier semigroup of the inverse semigroup $\cpK(\U)$ for a right inverse $T$-set $\U$.
This is an analogy of Kasparov's Theorem (see \cite[Theorem 2.4]{Lan95}).

\begin{theo}\label{Theorem: L(U) = L(K(U))}
	For a right inverse $T$-set $\U$, the multiplier semigroup $M(\cpK(\U))$ of $\cpK(\U)$ is isomorphic to $\adL(\U)$.
\end{theo}
\begin{proof}
	Apply Proposition \ref{Proposition: idealizer} to the case such that $S=\cpK(\U)$ and $\theta$ is the inclusion map.
	The idealizer of $\cpK(\U)$ in $\adL(\U)$ is nothing but $\adL(\U)$.
\end{proof}

\section{Relation to inverse Rees matrix semigroups}
\label{Section: Relation to inverse Rees matrix semigroups}

In the semigroup theory, the \emph{Rees matrix semigroups} are studied well (see \cite{McA81} or \cite{McA83}, for example).
For an inverse semigroup $T$ and a set $I$, Afara and Lawson introduced a \emph{McAlister function} $p\: I \times I \rightarrow T$ and an inverse semigroup $IM(T,I,p)$ called the \emph{inverse Rees matrix semigroup over $T$} in \cite{AL13}.
Using this inverse semigroup, they characterized the inverse semigroups which are Morita equivalent to given inverse semigroup (\cite[Theorem 3.5]{AL13}).
We reprove this in Corollary \ref{Corollary: AL} in terms of our inverse set theory.
In this section, we see McAlister functions from the perspective of inverse sets.
We first recall the definition of McAlister functions and inverse Rees matrix semigroups.

\begin{defi}
	For a set $I$ and an inverse semigroup $T$, a map $p \: I \times I \rightarrow T$ is a \emph{partial McAlister function} if the following conditions hold: 
	For all $i,j,k\in I$,
	\begin{enumerate}[(MF1)]
		\item $p_{i,i} \in E(T)$,
		\item $p_{i,i} p_{i,j} p_{j,j} = p_{i,j}$,
		\item $p_{i,j}^* = p_{j,i}$, and
		\item $p_{i,j} p_{j,k} \leq p_{i,k}$.
	\end{enumerate}
	If a map $p \: I \times I \rightarrow T$ satisfies (MF1)-(MF4) and 
	\begin{enumerate}
		\item[(MF5)] For every $e \in E(T)$, there exists $i \in I$ such that $e \leq p_{i,i}$,
	\end{enumerate}
	we call it a \emph{McAlister function}.
\end{defi}

McAlister studied functions of this kind in \cite{McA83}.
Afara and Lawson introduced the name ``McAlister function'' in \cite{AL13}.
The name ``partial McAlister function'' is introduced in this paper.

The inverse Rees matrix semigroup $IM(T,I,p)$ over $T$ is constructed as follows:
Let $T$ be an inverse semigroup, $I$ be a set, and $p \: I \times I \rightarrow T$ be a partial McAlister function.
The set 
\[
RM(T,I,p) := \{ (j, t, i ) \in I \times T \times I \mid p_{j,j} t p_{i,i} = t \}.
\]
becomes a regular semigroup with respect to the multiplication defined as
\[
(j_2,t_2,i_2)(j_1,t_1,i_1) := (j_2, t_2 p_{i_2,j_1} t_1, i_1)
\]
for $(j_1,t_1,i_1), (j_2,t_2,i_2) \in RM(T,I,p)$.
See \cite[Lemma 2.1 and 2.3]{AL13} for the proof of this fact.
This semigroup is called the \emph{regular Rees matrix semigroup}.

We define an equivalence relation $\gamma$ on $RM(T,I,p)$ by declaring that 
\[
(j_1,t_1,i_1) \gamma (j_2,t_2,i_2) :\Leftrightarrow p_{j_1,j_2} t_2 p_{i_2,i_1} = t_1 \text{ and } p_{j_2,j_1} t_1 p_{i_1,i_2} = t_2.
\]
We denote the quotient $RM(T,I,p)/\gamma$ as $IM(T,I,p)$ and the equivalence class of $(j,t,i)$ as $[j,t,i]$.
The set $IM(T,I,p)$ becomes an inverse semigroup with respect to the induced multiplication by $RM(T,I,p)$, that is,  
\[
[j_2,t_2,i_2][j_1,t_1,i_1] :=  [j_2, t_2 p_{i_2,j_1} t_1, i_1]
\]
for every $[j_1,t_1,i_1], [j_2,t_2,i_2] \in IM(T,I,p)$.
The generalized inverse of $[j,t,i]$ is $[i,t^*,j]$.
We remark that \cite[Lemma 2.6]{AL13} claims that $\gamma$ is the minimum inverse congruence on $RM(T,I,p)$.

We show that for every partial McAlister function $p\: I \times I \rightarrow T$, there exists an inverse right $T$-set $\U_p$ such that the inverse Rees matrix semigroup $IM(T,I,p)$ is isomorphic to the inverse semigroup $\cpK(\U_p)$. 

For a partial McAlister function $p \: I \times I \rightarrow T$, we define a set 
\[
\U_p' := \{ (j,t) \in I \times T \mid p_{j,j}t = t \}
\]
and a relation $\sim$ on $\U_p'$ by 
\[
(j_1,t_1) \sim (j_2,t_2) :\Leftrightarrow t_1 = p_{j_1,j_2}t_2 \text{ and } t_2 = p_{j_2,j_1}t_1.
\]

\begin{lemm}
	The relation $\sim$ on $\U_p'$ is an equivalence relation.
\end{lemm}
\begin{proof}
	For every $(j,t)\in\U_p'$, $(j,t) \sim (j,t)$ holds by $t = p_{j,j} t$.
	It is clear that the relation $\sim$ is symmetric.
	Take elements $(j_k,t_k) \in \U_p'$ for $k=1,2,3$ with $(j_k,t_k) \sim (j_{k+1},t_{k+1})$ for $k = 1,2$.
	By (MF4), we get 
	\begin{align*}
	t_1 
	&= p_{j_1,j_2}t_2\\
	&= p_{j_1,j_2}p_{j_2,j_3}t_3\\
	&\leq  p_{j_1,j_3}t_3\\
	&= p_{j_1,j_3}p_{j_3,j_2}t_2\\
	&\leq p_{j_1,j_2}t_2\\
	&= t_1.
	\end{align*}
	Thus we get $t_1  = p_{j_1,j_3}t_3$. 
	In a similar way, $t_3  = p_{j_3,j_1}t_1$ holds.
	Thus $(j_1,t_1) \sim (j_3,t_3)$.
\end{proof}

We denote the quotient set $\U_p'/\sim$ as $\U_p$ and the equivalence class of $(j,t)$ as $[j,t]$.
We define a right action of $T$ on $\U_p$ as 
\[
[ j_1, t_1 ]t := [ j_1, t_1t ]
\]
and a right pairing on $\U_p$ as
\[
\rin{ [ j_2, t_2 ] }{ [ j_1, t_1 ] }{\U_p} := t_2^*p_{j_2,j_1}t_1
\]
for $[j_1,t_1], [j_2,t_2] \in \U_p$ and $t \in T$.

\begin{lemm}\label{Lemma: Up is regular}
	The set $\U_p$ becomes a right regular $T$-set.
\end{lemm}
\begin{proof}
	We first show that the above structures are well-defined.
	Take elements $[ j_1, t_1 ]$ and $[ j'_1, t'_1 ]$ of $\U_p$ with $[ j_1, t_1 ] = [ j'_1, t'_1 ]$, that is, $t_1' = p_{j'_1,j_1}t_1$ and $t_1 = p_{j_1,j'_1}t'_1$.
	For every $t\in T$, these imply $t_1't = p_{j'_1,j_1}t_1t$ and $t_1t = p_{j_1,j'_1}t'_1t$.
	Thus we get $[ j_1, t_1t ] = [ j'_1, t'_1t ]$.
	This implies that the right action is well-defined.
	Take $[ j_k, t_k ], [ j'_k, t'_k ]\in \U_p$ with $[ j_k, t_k ] = [ j'_k, t'_k ]$ for $k = 1,2$.
	We get 
	\begin{align*}
	{t'_2}^*p_{j'_2,j'_1}t'_1 &= t_2^*p_{j_2,j_2'} p_{j'_2,j'_1} p_{j_1',j_1}t_1\\
	&\leq t_2^*p_{j_2,j_1}t_1\\
	&= {t'_2}^*p_{j_2',j_2} p_{j_2,j_1} p_{j_1,j_1'}t_1'\\
	&\leq {t'_2}^*p_{j_2',j_1'}t_1'.
	\end{align*}
	Thus $t_2^*p_{j_2,j_1}t_1 = {t'_2}^*p_{j'_2,j'_1}t'_1$ holds.
	This implies that the right pairing is well-defined.
	
	We can check easily that $\U_p\times T \rightarrow \U_p; ([j,t],t') \mapsto [j,tt']$ becomes a right action, and that (R-i) and (R-ii) holds for $\rin{\cdot}{\cdot}{\U_p}$.
	We get (R-iii) because 
	\[
	[j,t]\rin{ [j,t] }{ [j,t] }{\U_p} = [j, tt^*p_{j,j}t] = [j,t]
	\]
	holds for every $[j,t]\in \U_p$.
\end{proof}

We define a left action of $IM(T,I,p)$ on $\U_p$ as 
\[
[ j_2, t_2, i_2 ] [ j_1, t_1 ] := [ j_2 , t_2p_{i_2,j_1}t_1 ]
\]
for every $[ j_1, t_1 ] \in \U_p$ and $ [ j_2, t_2, i_2 ] \in IM(T,I,p)$, and a left pairing as 
\[
\lin{ [ j_2, t_2 ] }{ [ j_1, t_1 ] }{\U_p} := [ j_2, t_2t_1^*, j_1 ]
\]
for every $[ j_1, t_1 ], [j_2, t_2] \in \U_p$.
We can check that these structures are well-defined and form a left regular $IM(T,I,p)$-set in a similar way to Lemma \ref{Lemma: Up is regular}.
We can also check that the left and right actions and the left and right pairings are compatible respectively.
Hence $\U_p$ becomes a partial Morita equivalence from $IM(T,I,p)$ to $T$.

\begin{lemm}\label{Lemma: Up full}
	For a partial McAlister function $p\: I \times I \rightarrow T$, $\U_p$ becomes a partial Morita equivalence from $IM(T,I,p)$ to $T$ which is left full.
	If $p$ is a McAlister function, then $\U_p$ becomes a Morita equivalence.
\end{lemm}
\begin{proof}
	As mentioned above, $\U_p$ is a partial Morita equivalence from $IM(T,I,p)$ to $T$.
	For every $[j,t,i]\in IM(T,I,p)$, $\lin{ [j,t] }{ [i,t^*t] }{\U_p} = [j,t,i]$ holds.
	Thus $\U_p$ is left full.
	Assume $p$ is a McAlister function.
	Take $t \in T$.
	By the axiom (MF5), there exists $i \in I$ with $tt^* \leq p_{ii}$.
	Hence we have $\rin{ [i,tt^*] }{ [i,t] }{\U_p} = tt^*p_{ii}t = t$.
	Thus $\U_p$ is right full.
\end{proof}

\begin{coro}\label{Corollary: Up}
	For a partial McAlister function $p\: I \times I \rightarrow T$, $IM(T,I,p)$ is isomorphic to the inverse semigroup $\cpK(\U_p)$.
\end{coro}
\begin{proof}
	This follows from Lemma \ref{Lemma: Up full} and Corollary \ref{Corollary: partial Morita iff...} (i).
\end{proof}

Conversely, for a right inverse $T$-set $\U$, we obtain a partial McAlister function $p_\U$ as follows:
The following proposition is similar to \cite[Lemma 3.3]{AL13}.
\begin{prop}\label{Proposition: pairing MF}
	For a right inverse $T$-set $\U$, a map $p_{\U} \: \U \times \U \rightarrow T; (u,v) \mapsto \rin{ u }{ v }{\U}$ is a partial McAlister function.
	If the right pairing $\rin{\cdot}{\cdot}{\U}$ of $\U$ is right full, then $p_\U$ is a McAlister function.
\end{prop}
\begin{proof}
	It is easy to check that (MF1)-(MF3) hold. 
	For every $u,v,w\in\U$, we have
	\begin{align*}
	\rin{ u }{ v }{ \U }\rin{ v }{ w }{ \U } &= \rin{ u }{ v \rin{ v }{ w }{ \U } }{ \U } \\
	&=  \rin{ u }{ w \rin{ w }{ v }{ \U } \rin{ v }{ w }{ \U } }{ \U }\\
	&=  \rin{ u }{ w }{ \U } \rin{ w }{ v }{ \U } \rin{ v }{ w }{ \U }\\
	&\leq \rin{ u }{ w }{\U}.
	\end{align*}
	Thus $p_{\U}$ is a partial McAlister function.
	If the right pairing $\rin{\cdot}{\cdot}{\U}$ of $\U$ is full, then the $p_\U$ satisfies (MF5) obviously.
\end{proof}

\begin{lemm}\label{Lemma: UpU}
	For a right inverse $T$-set $\U$, the right inverse $T$-set $\U_{p_\U}$ associated with the partial McAlister function $p_\U$ defined in Proposition \ref{Proposition: pairing MF} is isomorphic to $\U$.
\end{lemm}
\begin{proof}
	For an inverse right $T$-set $\U$, we can check easily that the map $\U_{p_\U} \rightarrow \U; [u,t] \mapsto ut$ is an isomorphism between right inverse $T$-sets.
\end{proof}

\begin{coro}[{\cite[Theorem 3.5]{AL13}}]\label{Corollary: AL}
	Let $T$ be an inverse semigroup. 
	For every McAlister function $p\: I \times I \rightarrow T$, the inverse Rees matrix semigroup $IM(T,I,p)$ is Morita equivalent to $T$, and every inverse semigroup $S$ Morita equivalent to $T$ is isomorphic to one of this form.
\end{coro}
\begin{proof}
	For every McAlister function $p \: I \times I \rightarrow T$, $\U_p$ becomes a Morita equivalence from $IM(T,I,p)$ to $T$ by Lemma \ref{Lemma: Up full}.
	
	For a Morita equivalence $\U$ from $S$ to $T$, $S$ is isomorphic to $\cpK(\U)$ by Corollary \ref{Corollary: partial Morita iff...} (ii).
	The inverse semigroup $\cpK(\U)$ is isomorphic to $\cpK(\U_{p_\U})$ by Lemma \ref{Lemma: UpU}.
	This is isomorphic to $IM(T,\U,p_\U)$ by Corollary \ref{Corollary: Up}.
\end{proof}

\bibliography{reference}
\bibliographystyle{alpha}

\end{document}